\documentclass[oneside,reqno]{amsart}
\usepackage{amsthm,graphicx}
\usepackage{amssymb}
\usepackage[usenames]{color} 
\usepackage{colortbl}
\usepackage[shortlabels]{enumitem}
\usepackage{caption}
\usepackage{subcaption}
\usepackage{tikz}
\usepackage{pgfplots}
\usepgfplotslibrary{fillbetween}

\usepackage[pagebackref=true,colorlinks=true]{hyperref}
\renewcommand*{\backref}[1]{}
\renewcommand*{\backrefalt}[4]{[{\tiny%
    \ifcase #1 Not cited.%
          \or Cited on page~#2.%
          \else Cited on pages #2.%
    \fi%
    }]}

\usepackage{chngcntr}
\counterwithin{equation}{section}
\theoremstyle{definition}
 \newtheorem{theorem}{Theorem}
 \newtheorem{proposition}[equation]{Proposition}
 \newtheorem{definition}[equation]{Definition}
 \newtheorem{remark}[equation]{Remark}
 \newtheorem{lemma}[equation]{Lemma}
 \newtheorem{corollary}[equation]{Corollary}
 \newtheorem{example}[equation]{Example}
\newtheorem{question}[equation]{Question}

\usepackage{etoolbox,xstring,xspace}
\makeatletter

\AtBeginDocument{%
\let\origref\ref
\renewcommand*\ref[1]{%
  \origref{#1}\xlabel{#1}}
}
\newrobustcmd*\xlabel[1]{%
   \ifcsdef{siteref@doc@#1}{}{\csgdef{siteref@doc@#1}{,}}%
    \@bsphack%
    \begingroup
       \csxdef{siteref@doc@#1}{\csuse{siteref@doc@#1},\thepage}%
         \protected@write\@auxout{}%
        {\string\SiteRef{siteref@#1}{\csuse{siteref@doc@#1}}}%
     \endgroup
     \@esphack%
}

\newrobustcmd*\SiteRef[2]{\csgdef{#1}{#2}}

\newrobustcmd*\xref[1]{%
\ifcsundef{siteref@#1}{%
     \@latex@warning@no@line{Label `#1' not defined}
     }{%
    \begingroup
      \StrGobbleLeft{\csuse{siteref@#1}}{2}[\@tempa]\relax%
      \def\@tempb{}%
      \@tempcnta=0\relax%
      \@tempcntb=\@ne\relax%
      \def\do##1{\advance\@tempcnta\@ne}%
      \expandafter\docsvlist\expandafter{\@tempa}%
       \def\do##1{%
         \ifnum\@tempcntb=\@tempcnta\relax%
            \hyperpage{##1}%
         \else
            \hyperpage{##1},%
          \fi%
          \advance\@tempcntb\@ne
       }%
       [\expandafter\docsvlist\expandafter{\@tempa}]\xspace%
    \endgroup
   }%
}

\makeatother

\newcommand{\m}[1]{}

\newcommand{\used}[1]{(\text{Used\  on\  pages\  \xref{#1}\!\!){\ }}}
\newcommand{\laber}[1]{\label{#1} \used{#1}}

\newcommand\RR{\mathbb R}
\newcommand\QQ{\mathbb Q}
\newcommand\ZZ{\mathbb Z}
\newcommand\NN{\mathbb N}

\newcommand\A{\mathcal A}
\newcommand\B{\mathcal B}

\newcommand\e{\varepsilon}
\newcommand{\supp}{\mathrm{supp}}

\newcommand{\Add}{\mathrm{Add}}

\newcommand{\vv}{\bf{verified}}

\newcommand{\p}{{\bf p}}
\newcommand{\q}{{\bf q}}
\newcommand{\h}{\text{\small h}}
\newcommand{\C}{\text{C}}

\reversemarginpar

\numberwithin{equation}{section}

\usepackage{amsaddr}
\begin{document}

\date{\today}
\title{Introduction to tropical series and wave dynamic on them}

\author[N. Kalinin]{Nikita Kalinin}
\address{National Research University Higher School of Economics, Soyuza Pechatnikov str., 16, St. Petersburg, Russian Federation}

\author[M. Shkolnikov]{Mikhail Shkolnikov}
\address{Universit\'e de Gen\`eve, Section de
  Math\'ematiques, Route de Drize 7, Villa Battelle, 1227 Carouge, Switzerland}


\begin{abstract}
The theory of tropical series, that we develop here, firstly appeared in the study of the growth of pluriharmonic functions. Motivated by waves in sandpile models we introduce a dynamic on the set of tropical series, and it is experimentally observed that this dynamic obeys a power law. So, this paper serves as a compilation of results we need for other articles and also introduces several objects interesting by themselves.
\end{abstract}
\maketitle

\subsection{Motivation and history}
The purpose of this article is two-fold: we develop the theory of a dynamic on tropical series on two-dimensional domains and prove ancillary statements for our works about sandpile dynamic (see \cite{us}, \cite{announce} for details), non-commutative toric varieties (see \cite{mishathesis} for a sketch), and number $\pi$ (see \cite{pi_short}). Our initial motivation was \cite{firstsand} where it was experimentally observed that tropical curves appear in sandpile models and behave nicely when we add more sand. We think that this text has an independent interest, so we separated it from \cite{us} to make it more accessible for a general audience. 

We experimentally observed, \cite{sandcomputation}, that the dynamic on the space of tropical series, proposed here, obeys power law. Namely, the distribution of the area of an avalanche (a direct analog of that for sandpiles) in this model has the density function of the type $p(x) = cx^{\alpha}$. To the best of our knowledge, this simple geometric dynamic is the only model, among the ways to obtain power laws in a simulation, which produces a continuous random variable.

Tropical series appeared in the works of Christer Oscar Kiselman, \cite{MR743626}, while studying the growth of plurisubharmonic functions, see also \cite{kiselman2014questions}, Section 5, and \cite{Abakumov2017}. Tropical series in one variable can be studied in the context of ultradiscretization of differential equations, see \cite{tohge2014order} and references therein. See also \cite{MR2482129},\cite{MR2795727},\cite{korhonen2015tropical} for tropical Nevalinna theory. One-dimensional tropical series are used in automata-theory: \cite{lahaye2015compositions}, \cite{lombardy2006sequential}.

For a general introduction to tropical geometry, see \cite{BIMS}, \cite{mikh2}, or \cite{Brug}. It seems that the results of this paper can be extended to higher dimensions, but the proofs will be technically more involved.

\subsection{Plan of this paper, main objects.}
A tropical series on a closed convex domain $\Omega\subset\RR^2$ is just a function which is locally a minimum of a finite number of functions $ix+jy+a_{ij}$ where $i,j\in\ZZ, a_{ij}\in \RR$. For the purpose of \cite{us} we need to consider tropical series which are non-negative. Clearly, non-negative tropical series on $\RR^2$ are just constants, so we restrict our attention to {\it admissible} $\Omega$'s, see Definition~\ref{def_omegaadmissible}. We study general properties of tropical series in Sections~\ref{sec_series},~\ref{sec_omegaseries}. Somewhat the main tropical series associated to $\Omega$ is the {\it weighted distance function}, see Section~\ref{sec_distance}. 

In Section~\ref{sec_gp} we define the main character of this paper, the ``wave'' operator $G_\p$, where $\p\in\Omega^\circ$ and $G_\p$ acts on tropical series, and study its properties. The word ``wave'' stands for the fact that $G_\p$ is the scaling limit incarnation of sending waves from $\p$ in a sandpile model, see \cite{us} for details. Section~\ref{sec_dynamic} is devoted to the dynamic generated by applying operators $G_\p$ for different points. In Section~\ref{sec_lift} we show how to lift $G_\p$ to an operator on Laurent polynomials over a field of characteristic two. Its algebraic meaning is yet to be discovered.

In Sections~\ref{sec_qpolygons},~\ref{sec_exhausting} we show how to approximate $\Omega$ by a somewhat canonical family of $\QQ$-polygons, i.e. (possibly non-compact) polygons with finite number of sides of rational slopes.

Sections~\ref{sec_nice},~\ref{sec_coarse} further reduce the study of dynamic to the case of so called {\it nice} tropical series, which behave well near the boundary of $\Omega$. Proposition~\ref{prop_nice} tells that in order to approximate $G=G_{\p_1}G_{\p_2}\dots 0_\Omega$, by so-called blow-ups we can restrict the dynamic to a $\QQ$-polygon close to $\Omega$. Proposition~\ref{prop_smooth} asserts that by changing $G$ just a bit we may assume that all tropical curves are smooth (Definition~\ref{def_smoothvertex}) during this dynamic. We summarize these results in Section~\ref{sec_summary}.

The tropical curve $C(G)$ for the above function $G$ has the minimal {\it tropical symplectic area} (defined in \cite{tony}) among the curves passing through the points $\p_1,\dots,\p_n$, see Section~\ref{sec_area} where we also explain the name of this notion. 

Smooth tropical curves corresponding to nice tropical series are main objects in the proofs in \cite{us}, \cite{announce}. Using results of this paper we will reduce the theorems in \cite{us} to the local questions which can be addressed purely combinatorially with help of super-harmonic functions \cite{us_solitons}. 

\subsection{Acknowledgments}

We thank Andrea Sportiello for sharing his insights on perturbative
regimes of the Abelian sandpile model which was the starting
point of our work on sandpiles. Our proofs required developing the theory of tropical series, presented here.

The first author, Nikita Kalinin, is funded by SNCF PostDoc.Mobility grant 168647. 
Support from the Basic Research Program of the National Research University Higher School of Economics is gratefully acknowledged.
 

The second author, Mikhail Shkolnikov, is supported in part by the grant 159240 of the Swiss National Science
Foundation as well as by the National Center of Competence in Research
SwissMAP of the Swiss National Science Foundation. 

\section{Tropical series}
\label{sec_series}
Recall that a tropical Laurent polynomial (later just {\it tropical polynomial})  $f$ on $U\subset\RR^2$ in two variables is a function  $f: U\to \RR$ which can be written as \begin{equation}\label{eq_Ftroppoli}
f(x,y)=\min_{(i,j)\in \A} (ix+jy+a_{ij}), a_{ij}\in\RR
\end{equation} where $\A$ is
a {\bf finite} subset of $\ZZ^2.$ Each point $(i,j)\in\A$ corresponds to a {\it monomial} $ix+jy+a_{ij}$, the number $a_{ij}$ is
called {\it the coefficient} of $f$ of the monomial corresponding to the point $(i,j)\in\A$.  The locus of the points where a tropical polynomial $f$ is not
smooth is a {\it tropical curve} (see \cite{mikh2}). We denote this locus by $C(f)\subset U$.

\begin{definition}
\laber{def_tropseries}
Let $U\subset\RR^2, U^\circ\ne\varnothing$. A continuous function $f:U\to\RR$ is called {\it a tropical series} if for each $(x_0,y_0)\in U^\circ$ there exists an open neighborhood $W\subset U$ of $(x_0,y_0)$ such that $f|_W$ is a tropical polynomial. 
\end{definition}

\begin{definition}[Cf. Definition~\ref{def_tropicalanalitycal}]
A {\it tropical analytic curve} in $U$ is the locus of non-linearity
of a tropical series $f$ on $U^\circ$. We denote this curve by $C(f)\subset U^\circ$.
\end{definition}

\begin{example}  Tropical $\Theta$-divisors \cite{mikhalkin2006tropical} are tropical analytic curves in $\RR^2$, as well as the standard grid -- the union of all horizontal and vertical lines passing through lattice points, i.e. the set $$C=\bigcup\limits_{k\in \ZZ}\{(k,y)|y\in\RR\}\cup \{(x,k)|x\in\RR\}.$$ 
\end{example}

The following example illustrates that a tropical series on $\Omega^\circ$ in general cannot be
extended to $\partial\Omega$. 
\begin{example}
Consider a tropical analytic curve $C$ in the square $(0,1]\times[0,1]$, presented as $$C=\bigcup_{n\in\NN}\Bigl\{\big(1/n,y\big)|y\in[0,1]\Bigl\}\cup \Bigl\{\big(x,1/2\big)|x\in(0,1]\Bigl\}.$$
For all tropical series $f$ with $C(f)=C$, the sequence of values of $f(x,y)$ tends to $-\infty$ as $x\to 0$.
\end{example}

\begin{question}
When can we extend a tropical series from $\Omega^\circ$ to $\partial\Omega$? 
\end{question}

Tropical series on non-convex domains exhibit the behavior as in the following example.

\begin{example}
The function $f(x,y)=\min(3,x+[y])$ is a tropical series on the following $U$: $$U_1=\big([0,5]\times[0,1]\big)\cup \big([4,5]\times [1,2]\big), U_2=\big([0,5]\times[2,3]\big)\cup \big([4,5]\times [1,2]\big), U=(U_1\cup U_2)^\circ,$$ but $f|_{U_1^\circ}=\min (3, x), f|_{U_2^\circ}=\min(3,x+2)$ and the monomial $x$ appears with different coefficients $0,2$ in the different parts of $U$. 
\end{example}

\begin{definition}
\laber{def_omegaadmissible} 
A convex closed subset $\Omega\subset\RR^2$ is said to be {\it not admissible} if one of the following cases takes place:
\begin{itemize}
\item $\Omega$ has empty interior $\Omega^\circ$ (i.e. $\Omega$ is a subset of a line),
\item $\Omega$ is $\RR^2$,
\item $\Omega$ is a half-plane with the boundary of irrational slope,
\item $\Omega$ is a strip between two parallel lines of irrational slope.
\end{itemize} 
Otherwise, $\Omega$ is called {\it admissible}.
\end{definition}

\begin{definition}\laber{def_singlesupport}Let $\Omega\subset \RR^2$. For $(i,j)\in\mathbb{Z}^2$ denote by $c_{ij}\in\RR\cup\{-\infty\}$ the infimum of $ix+jy$ over $(x,y)\in\Omega.$ Let $\A_\Omega$ be the set of pairs $(i,j)$ with $c_{ij}\neq -\infty.$ Note that if $\Omega$ is bounded, then $\A_{\Omega}=\ZZ^2$. For each $(i,j)\in \A_\Omega$ we define 
\[\label{eq_lij}
l^{ij}_\Omega(x,y)=ix+jy-c_{ij}.
\]
\end{definition}

Note that $l^{ij}_\Omega$ is positive on $\Omega^\circ$. Also, $\A_{\Omega}$ always contains $(0,0)$.

\begin{proposition}
\laber{prop_admissible}
A convex closed set $\Omega$ is admissible if and only if $\A_{\Omega}\ne\{(0,0)\}$ and $\Omega^\circ\ne \varnothing$. 
\end{proposition}
\begin{proof}
It is easy to verify that if $\Omega$ is not admissible, then $\Omega^\circ=\varnothing$ or $\A_{\Omega}=\{(0,0)\}$. Let us prove ``only if'' direction. Since $\Omega\ne\RR^2$, there exists a boundary point $z$ of $\Omega$ and a support line $l$ at $z$. If the slope of $l$ is rational, then $\A_{\Omega}$ contains the corresponding lattice point; if this slope is irrational but $l$ does not belong to the boundary of $\Omega$, then there exists another support line of $\Omega$ with a close rational slope. So, we may suppose that $l$ is contained in the boundary of $\Omega$. If there is no other boundary points of $\Omega$, then $\Omega$ is a half-plane and is not admissible. If there exists another boundary point in $\partial\Omega\setminus l$, then we repeat the above arguments and find a support line of $\Omega$ with rational slope. Another case, i.e. that $\Omega$ is a strip between two lines of the same irrational slope, is not possible since $\Omega$ is admissible.
\end{proof}

\section{$\Omega$-tropical series}
\label{sec_omegaseries}
From now on we always suppose that $\Omega$ is an {\bf admissible} convex closed subset of $\RR^2$.

\begin{definition}
\laber{def_tropicalanalitycal} 
An $\Omega$-{\it tropical series} is a function
$f:\Omega\to\RR_{\geq 0}$, $f|_{\partial\Omega}= 0$,  such that
\begin{equation}
\label{eq_omegatrop}
f(x,y)=\inf\limits_{(i,j)\in\A}(ix+jy+a_{ij}), a_{ij}\in\RR,
\end{equation}
 and
$\A\subset\ZZ^2$ is not necessary finite. An $\Omega$-{\it
tropical analytic curve} $C(f)$ on $\Omega^\circ$ is the corner locus
(i.e. the set of non-smooth points) of an $\Omega$-tropical series on $\Omega^\circ$. 
\end{definition}
The reason to consider only admissible sets is Proposition~\ref{prop_admissible} is that an $\Omega$-tropical analytic curve on non-admissible $\Omega$ is always the empty set, because either $\Omega^\circ$ is empty or the only $\Omega$-tropical series is the function $0$.

\begin{question} An $\Omega$-tropical series can be thought of an analog of a series  $f_t(x,y)=\sum_{(i,j)\in \A_\Omega}t^{a_{ij}} x^iy^j$ with $t\in\RR_{>0}$ very small. Is is true that $\Omega$ is the limit of the images of the region of convergence of $f_t$ under the map $\log_t: (x,y)\to (\log_t|x|,\log_t|y|)$, and the corresponding $\Omega$-tropical analytic curve is the limit of the images of $\{f_t(x,y)=0\}$ under $\log_t|\cdot|$ when $t\to 0$?
\end{question}

\begin{lemma}
\laber{lemma_estimate} Let $U\subset\RR^2$ be an open set and $K\subset U$ be a compact set. For any  $C>0$ the set $$\mathcal{M}=\big\{(i,j)\in\ZZ^2| \exists  d\in \RR, (ix+jy+d)|_U\geq 0, \exists (x_0,y_0)\in K, (ix_0+jy_0+d)\leq C\big\},$$
i.e. the set of monomials which potentially can contribute on $K$ to an $\Omega$-tropical function $f$ with $\max_K f\leq C$, is finite.
\end{lemma}
\begin{proof} If $U=\RR^2$, $\mathcal{M}=\{(0,0)\}$. So, let $R>0$ denote the distance between $K$ and $\RR^2\backslash U.$ Then $(ix+jy+d)|_K\geq R\cdot\sqrt{i^2+j^2}$ for any $i,j$ and $d$ such that $(ix+jy+d)|_U\geq 0.$ Therefore, $i^2+j^2\leq C^2R^{-2}$ for all $(i,j)\in\mathcal{M}.$
\end{proof}

\begin{lemma}
\laber{lemma_welldefinedseries}
In the definition of an $\Omega$-tropical series $f$, \eqref{eq_omegatrop}, we can 
replace ``$\inf$'' by ``$\min$'', i.e. at every point $(x,y)\in\Omega^\circ$ we have $$\inf\limits_{(i,j)\in\A}(ix+jy+a_{ij})=\min\limits_{(i,j)\in\A}(ix+jy+a_{ij}).$$
\end{lemma}

\begin{proof}
\label{proof_welldefinedseries}
Suppose that for a point $(x_0,y_0)\in\Omega^\circ$ and  for each $(i,j)\in\A$ the value of the monomial $a_{ij}+ix_0+jy_0$ is distinct from the value of the infimum  $$\inf\limits_{(i,j)\in\A}(ix_0+jy_0+a_{ij}).$$  Thus, there exists $\C>0$ such that we have $a_{ij}+ix_0+jy_0<\C$ for infinite number of monomials $(i,j)\in\A$.  Since $(a_{ij}+ix+jy)|_\Omega\geq 0$ for all $(i,j)\in\A$, applying Lemma~\ref{lemma_estimate} yields a contradiction. 
\end{proof}
At a point on $\partial\Omega$ where there is no tangent line with a rational slope we actually have to take the infimum, cf. the proof of Lemma~\ref{lemma_lomegaiszero}. Similarly, applying Lemma~\ref{lemma_estimate} for small compact neighbors of points we obtain the following result.

\begin{corollary}
An $\Omega$-tropical series (Definition~\ref{def_tropicalanalitycal}) is a tropical series on $\Omega$ in the sense of Definition~\ref{def_tropseries}.
\end{corollary}

Note that an $\Omega$-tropical series $f:\Omega\to\RR$ always has different presentations as the minimum of linear functions. For example, if $\Omega$ is the square $[0,1]\times[0,1]\subset\RR^2$, then $\min(x,1-x,y,1-y,1/3)$ equals at every point of $\Omega$ to $\min(x,1-x,y,1-y,1/3, 2x,5-2x)$. 
\begin{definition}
[cf. \cite{kiselman2014questions}, Lemma 5.3] \laber{def_canonicalseries}
To resolve this ambiguity, we suppose that, in $\Omega^\circ$, a tropical series $f$ is always (if the opposite is not stated explicitly) given by 
\begin{equation}
\laber{eq_series}
f(x,y)=\min\limits_{(i,j)\in\A}(ix+jy+a_{ij})
\end{equation}
 with $\A=\A_\Omega$ (Definition~\ref{def_singlesupport}) and with as minimal as possible coefficients $a_{ij}$. We call this presentation {\it the canonical form} of a tropical series. For each $\Omega$-tropical series there exists a unique canonical form. 
\end{definition}

\begin{example}
\label{ex_bigform}
The canonical form of $\min(x,1-x,y,1-y,1/3)$ on $\Omega=[0,1] \times [0,1]$ is $f(x,y)$ as in \eqref{eq_series} with $\A=\ZZ^2$, $a_{00}=1/3$ and $a_{ij}  = -\min_{(x,y)\in\Omega}(ix+jy)$ for $(i,j)\in\ZZ^2\setminus\{(0,0)\}$. 
\end{example}
\begin{proof}
It is easy to check that $f(x,y)=\min(x,1-x,y,1-y,1/3)$ on $\Omega$. All the coefficients $a_{ij},(i,j)\ne (0,0)$ are chosen as minimal with the condition that $ix+jy+a_{ij}$ is non-negative on $\Omega$. Finally, in the canonical form of $\min(x,1-x,y,1-y,1/3)$ the coefficient $a_{00}$ can not be less than $1/3$.
\end{proof}

\begin{lemma}
\laber{lemma_usualisomegatropical}
Suppose that $\Omega$ is admissible and a continuous function $f:\Omega\to\RR$ satisfies two conditions: 1) $f|_{\Omega^\circ}$ is a tropical series, and 2) $f|_{\partial\Omega}=0$. Then $f$ is an $\Omega$-tropical series (Definition~\ref{def_tropicalanalitycal}).
\end{lemma}

\begin{proof}
Let $f|_U=ix+jy+a_{ij}$ for an open $U\subset\Omega^\circ$. It follows from convexity of $\Omega$ and local concaivity of $f$ that $f(x,y)\leq ix+jy+a_{ij}$ on $\Omega$. Therefore in $\Omega^\circ$ we have
$$f(x,y) = \min \{ix+jy+a_{ij}| (i,j,a_{ij}), \exists \text{ open } U\subset \Omega^\circ,
 f(x,y)|_U=ix+ij+a_{ij}\}.$$
\end{proof}

\section{Tropical distance function}
\label{sec_distance}

\begin{definition}
\laber{def_generalweighteddistance}
We use the notation of \eqref{eq_lij}. The {\it weighted distance function} $l_\Omega$ on $\Omega$ is defined by 
$$l_\Omega(x,y)=\inf \big\{l^{ij}_\Omega(x,y)| (i,j)\in \A_\Omega\setminus \{(0,0)\} \big\}.$$
\end{definition}



An example of a tropical analytical curve defined by $l_\Omega$ is drawn on the right hand side of Figure~\ref{fig_circle}.

\begin{remark}
\laber{rem_lomegaestimate}
If $f(x,y)=ix+jy+a_{ij},(i,j)\in\ZZ^2\setminus \{(0,0)\}, a_{ij}\in\RR$, $f|_\Omega\geq 0$, then $f\geq l_\Omega$ on $\Omega$.
\end{remark}

The same argument in the proof of Lemma~\ref{lemma_welldefinedseries} proves the following lemma.
\begin{lemma}
The function $l_\Omega$ is a tropical series in $\Omega^\circ$ (Definition~\ref{def_tropseries}). 
\end{lemma}

\begin{figure}[htbp]
\label{fig_twodiscsonepoint}
  \includegraphics[width=0.3\textwidth]{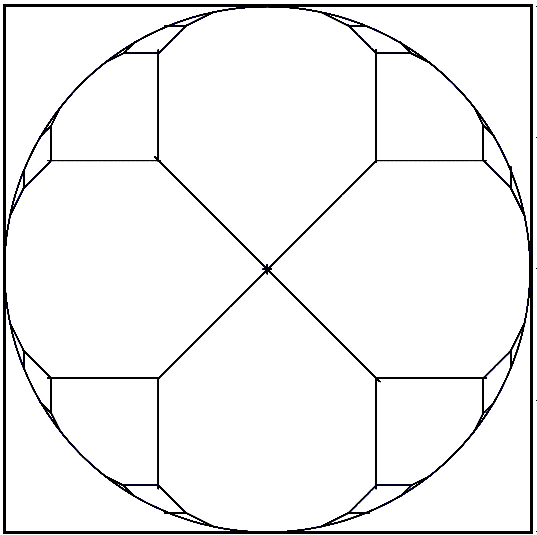}\
 \includegraphics[width=0.3\textwidth]{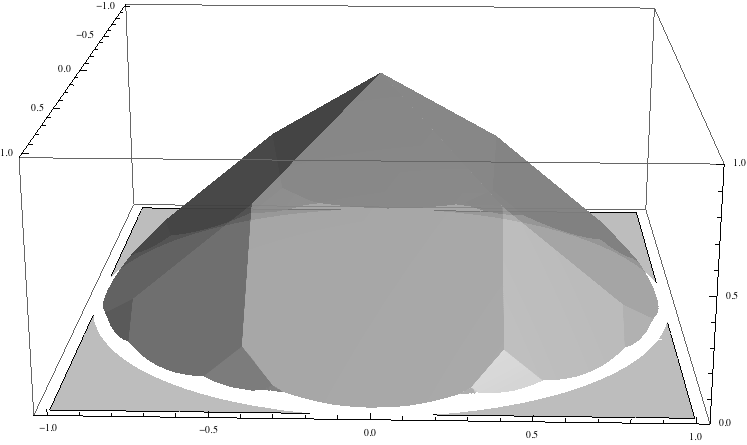}\
\caption{The central picture shows the corner locus of the right picture which is $l_{\Omega}$ (Definition~\ref{def_generalweighteddistance}) for $\Omega=\{x^2+y^2\leq 1\}$.}
  \label{fig_circle}
\end{figure}

\begin{lemma}
\laber{lemma_lomegaiszero}
The function $l_\Omega$ is an $\Omega$-tropical series.
\end{lemma}

\begin{proof}
\label{proof_lomegaiszero}
It is enough to prove that  $l_\Omega$ is zero on $\partial\Omega$ and continuous when we approach $\partial\Omega$.
It is clear that $l_\Omega=0$ on the points of $\{l^{ij}_\Omega=0\}\cap\partial\Omega$ for all $(i,j)\in\A_\Omega$. Suppose that there exists a point $(x_1,y_1)=z\in\partial\Omega$ where the only support line $L$ is of irrational slope $\alpha$. 

Since $\Omega$ is admissible, there exists a point $z'=(x_2,y_2)\in L$ which does not belong to $\partial\Omega$. Using continued fractions for $\alpha$, we get two sequences of numbers, $p_{2n}/q_{2n}<\alpha<p_{2n+1}/q_{2n+1}$ such that $$|\alpha-p_{m}/q_{m}|<1/q_{m}^2, \text{\ for\ all\ $m$,}$$ and $q_m$ tends to infinity.
Either for all even $k$, or for all odd $k$ the line through $z'$ with the slope $p_k/q_k$ does not intersect $\Omega$, so $(-p_k,q_k)\in\A_\Omega$ or $(p_k,-q_k)\in\A_\Omega$.  Thus, for such $k$ the absolute value of the linear function $$l_k(x,y)=q_i\left(y-\frac{p_k}{q_k}x- \big(y_2-\frac{p_k}{q_k}x_2\big)\right)$$ estimates $l^{-p_k,q_k}(x,y), (x,y)\in\Omega^\circ$ from above. The absolute value of $l_k$ at $z$ is 
\begin{align*}
|-p_kx_1 +q_ky_1+ (p_kx_2-q_ky_2)| = &|p_k(x_2-x_1)+q_k(y_1-y_2)| = \\
&|(x_1-x_2) (-p_k+q_k\alpha)|= |(x_1-x_2)(\alpha-p_k/q_k)q_k|\leq \left|\frac{x_1-x_2}{q_k}\right|,
\end{align*}

which tends to zero as $i\to\infty$. Therefore, we can construct a sequence of functions $l_k,k\to\infty$, whose values at $z$ tend to zero, and $l_{\Omega}\leq |l_k|$. This proves both continuity of $l_{\Omega}$ at $z$ and that $l_{\Omega}(z)=0$ for all $z\in\partial\Omega$.
 
%

\end{proof}

\section{Wave operators $G_\p$}
\label{sec_gp}

Recall that $\Omega\subset \RR^2$ is admissible (Definition~\ref{def_omegaadmissible}). Let $P=\{\p_1,\dots,\p_n\}$ be  a
finite collection of points in $\Omega^\circ$. Let $g$ be an $\Omega$-tropical series. 
\begin{definition}
\laber{def_vOmegaH}
Denote by $V(\Omega,P,f)$ the set of $\Omega$-tropical series $g$
such that $g|_{\Omega}\geq f$ and $g$ is not smooth at each of the points $\p\in P.$
\end{definition}
\begin{lemma}
\laber{lemma_Visnonempty}
The set $V(\Omega,P,f)$  is not empty. 
\end{lemma}
\begin{proof} Since $\Omega$ is admissible, $l_\Omega$ is well defined, and the function $$f'(z)=f(z)+\sum_{\p\in P}\min(l_\Omega(z),l_\Omega(\p))$$ belongs to $V(\Omega,P,f)$.  
\end{proof}
Clearly, if $f\geq g$ then $V(\Omega,P,f)\subset V(\Omega,P,g)$.

\begin{definition}
\laber{def_gp}
For a finite subset $P$ of $\Omega^\circ$ and an $\Omega$-tropical series $f$ we define an operator $G_P
$, given by $$G_P f(z)=\inf \{g(z)|g\in V(\Omega,P,f)\}.$$ If $P$ contains only one point $\p$ we write $G_\p$ instead of $G_{\{\p\}}$. 
\end{definition}
  
\begin{lemma}\laber{lem_gpmponotone}
Let $g$ and $f$ be two tropical series on $\Omega^\circ$ such that $g\leq f$ and $P\subset\Omega^\circ.$ Then $G_P g\leq G_P f$.
\end{lemma}

\begin{proof}
Indeed, $G_P f\geq f\geq g$ and $G_P f$ is not smooth at $P.$ Therefore, $G_P g\leq G_P f$ by definition of $G_P g.$ 
\end{proof}

  \begin{figure}[h]
    \centering
   \begin{tikzpicture}[scale=0.9]
    \centering
\draw[thick, dashed](0,0)--++(0,2)--++(1.4,0)--++(0.6,-0.6)--++(0,-1.4)--++(-2,0);

\draw[very thick](0,0)--++(0.4,0.4)--++(0,1.2)--++(1,0)--++(0.2,-0.2)--++(0,-1)--++(-1.2,0);
\draw[very thick](0,2)--++(0.4,-0.4);
\draw[very thick](1.4,2)--++(0,-0.4);
\draw[very thick](2,1.4)--++(-0.4,0);
\draw[very thick](2,0)--++(-0.4,0.4);
\draw(0.4,1)node{$\bullet$};
\draw(0.4,1)node[right]{$p$};

\begin{scope}[xshift=110]
\draw[thick, dashed](0,0)--++(0,2)--++(1.4,0)--++(0.6,-0.6)--++(0,-1.4)--++(-2,0);

\draw[very thick](0,0)--++(0.6,0.6)--++(0,0.8)--++(0.8,0)--++(0,-0.8)--++(-0.8,0);
\draw[very thick](0,2)--++(0.6,-0.6);
\draw[very thick](1.4,2)--++(0,-0.6);
\draw[very thick](2,1.4)--++(-0.6,0);
\draw[very thick](2,0)--++(-0.6,0.6);
\draw(0.6,1)node{$\bullet$};
\draw(0.6,1)node[right]{$p$};
\end{scope}

\begin{scope}[xshift=220]
\draw[thick,dashed](0,0)--++(0,2)--++(1.4,0)--++(0.6,-0.6)--++(0,-1.4)--++(-2,0);

\draw[very thick](0,0)--++(0.8,0.8)--++(0,0.4)--++(0.4,0)--++(0,-0.4)--++(-0.4,0);
\draw[very thick](0,2)--++(0.8,-0.8);
\draw[very thick](1.4,2)--++(0,-0.6);
\draw[very thick](2,1.4)--++(-0.6,0);
\draw[very thick](2,0)--++(-0.8,0.8);
\draw[very thick](1.4,1.4)--++(-0.2,-0.2);
\draw(0.8,1)node{$\bullet$};
\draw(0.8,1)node[left]{$p$};
\end{scope}

\begin{scope}[xshift=330]
\draw[thick,dashed](0,0)--++(0,2)--++(1.4,0)--++(0.6,-0.6)--++(0,-1.4)--++(-2,0);
\draw[very thick](2,0)--++(-2,2);
\draw[very thick](1.4,2)--++(0,-0.6);
\draw[very thick](2,1.4)--++(-0.6,0);
\draw[very thick](1.4,1.4)--++(-1.4,-1.4);
\draw(1,1)node{$\bullet$};
\draw(1,0.65)node{$p$};
\end{scope}

\begin{scope}[yshift=-80]
\draw[very thick](0,1)--++(1,1)--++(1,0)--++(0,-1)--++(-1,-1)--++(-1,1);
\draw[thick](0,1)--++(2,0);
\draw[thick](1,0)--++(0,2);
\draw[thick](1,1)--++(1,1);
\draw(0,1)node{$\bullet$};
\draw(1,0)node{$\bullet$};
\draw(2,1)node{$\bullet$};
\draw(1,2)node{$\bullet$};
\draw(2,2)node{$\bullet$};
\draw(1,1)node{$\bullet$};

\begin{scope}[xshift=110]
\draw[very thick](0,1)--++(1,1)--++(1,0)--++(0,-1)--++(-1,-1)--++(-1,1);
\draw[thick](0,1)--++(2,0);
\draw[thick](1,0)--++(0,2);
\draw(0,1)node{$\bullet$};
\draw(1,0)node{$\bullet$};
\draw(2,1)node{$\bullet$};
\draw(1,2)node{$\bullet$};
\draw(2,2)node{$\bullet$};
\draw(1,1)node{$\bullet$};
\end{scope}

\begin{scope}[xshift=220]
\draw[very thick](0,1)--++(1,1)--++(1,0)--++(0,-1)--++(-1,-1)--++(-1,1);
\draw[thick](0,1)--++(2,0);
\draw[thick](1,0)--++(0,2);
\draw[thick](1,2)--++(1,-1);
\draw(0,1)node{$\bullet$};
\draw(1,0)node{$\bullet$};
\draw(2,1)node{$\bullet$};
\draw(1,2)node{$\bullet$};
\draw(2,2)node{$\bullet$};
\draw(1,1)node{$\bullet$};
\end{scope}

\begin{scope}[xshift=330]
\draw[very thick](0,1)--++(1,1)--++(1,0)--++(0,-1)--++(-1,-1)--++(-1,1);
\draw[thick](1,2)--++(1,-1);
\draw(0,1)node{$\bullet$};
\draw(1,0)node{$\bullet$};
\draw(2,1)node{$\bullet$};
\draw(1,2)node{$\bullet$};
\draw(2,2)node{$\bullet$};
\draw(1,1)node{$\bullet$};
\draw (1,2)node[left]{$(0,-1)$};
\draw (1,0)node[left]{$(0,1)$};
\draw (0,1)node[left]{$(1,0)$};
\draw (2,1)node[right]{$(0,-1)$};
\draw (2,2)node[right]{$(-1,-1)$};
\draw (1,1)node[below]{$(0,0)$};

\end{scope}
\end{scope}

\end{tikzpicture}
    \caption{ First row shows how curves given by $G_p 0_\Omega$ depend on the position of the point in the pentagon $\Omega$. The second row shows monomials in their minimal canonical form. Note that the coordinate axes of the second row are actually reversed. Each lattice point on a below picture represents a face where the corresponding monomial is dominating on a top picture, see the bottom-right picture.} 
    \label{fig_onetroppoint}
\end{figure}
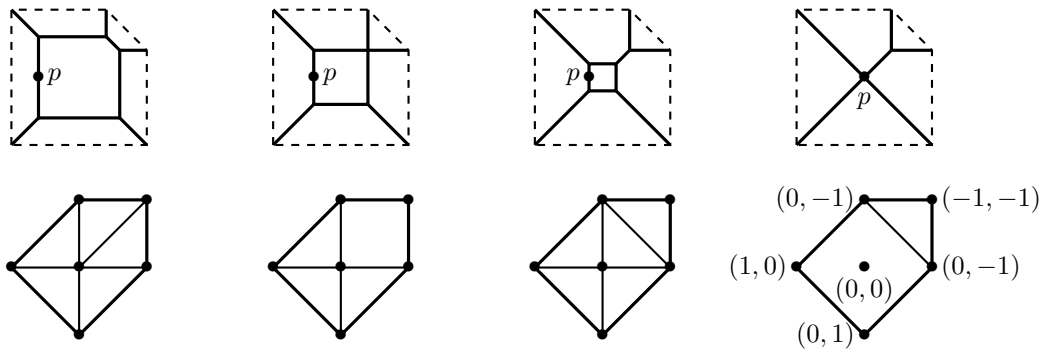
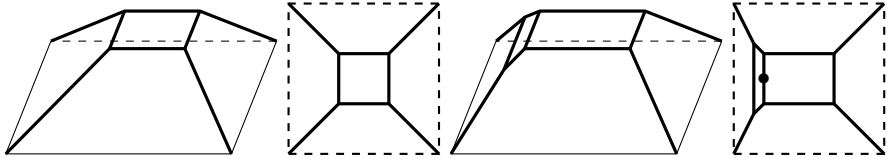
\begin{figure}[htbp]
\begin{tikzpicture}
[x={(0.5cm,0cm)}, y= {(0.1cm,0.25cm)}, z={(0.09cm,0.45cm)}, scale=6]
\draw (0,1,0)--(0,0,0)--(1,0,0)--(1,1,0);
\draw[dashed](1,1,0)--(0,1,0);
\draw[very thick] (0,0,0)--(1/3,1/3,1/3)--(2/3,1/3,1/3)--(1,0,0);
\draw[very thick] (0,1,0)--(1/3,2/3,1/3)--(2/3,2/3,1/3)--(1,1,0);
\draw[very thick] (1/3,1/3,1/3)--(1/3,2/3,1/3);
\draw[very thick] (2/3,1/3,1/3)--(2/3,2/3,1/3);
\end{tikzpicture}
\begin{tikzpicture}[scale=2]
\draw[very thick] (0,0)--(1/3,1/3)--(2/3,1/3)--(1,0);
\draw[very thick] (0,1)--(1/3,2/3)--(2/3,2/3)--(1,1);
\draw[very thick] (1/3,1/3)--(1/3,2/3);
\draw[very thick] (2/3,1/3)--(2/3,2/3);
\draw[thick, dashed](0,0)--(1,0)--(1,1)--(0,1)--cycle;
\end{tikzpicture}
\begin{tikzpicture}
[x={(0.5cm,0cm)}, y= {(0.1cm,0.25cm)}, z={(0.09cm,0.45cm)}, scale=6]
\draw (0,1,0)--(0,0,0)--(1,0,0)--(1,1,0);
\draw[dashed](1,1,0)--(0,1,0);
\draw[very thick] (0,0,0)--(2/15,4/15,4/15)--(1/5,1/3,1/3)--(2/3,1/3,1/3)--(1,0,0);
\draw[very thick] (0,1,0)--(2/15,11/15,4/15)--(1/5,2/3,1/3)--(2/3,2/3,1/3)--(1,1,0);
\draw[very thick] (1/5,1/3,1/3)--(1/5,2/3,1/3);
\draw[very thick] (2/3,1/3,1/3)--(2/3,2/3,1/3);
\draw[very thick] (2/15,4/15,4/15)--(2/15,11/15,4/15);
\end{tikzpicture}
\begin{tikzpicture}[scale=2]
\draw[very thick] (1,0)--(2/3,1/3)--(1/5,1/3)--(2/15,4/15)--(0,0);
\draw[very thick] (1,1)--(2/3,2/3)--(1/5,2/3)--(2/15,11/15)--(0,1);
\draw[very thick] (2/3,1/3)--(2/3,2/3);
\draw[very thick] (1/5,1/3)--(1/5,2/3);
\draw[very thick] (2/15,4/15)--(2/15,11/15);
\draw[thick, dashed] (0,0)--(1,0)--(1,1)--(0,1)--cycle;
\draw (1/5,1/2) node {$\bullet$};
\end{tikzpicture}
\caption{On the left: $\Omega$-tropical series $\min(x,y,1-x,1-y,1/3)$ and the corresponding tropical curve. On the right: the result of applying $G_{(\frac{1}{5},\frac{1}{2})}$ to the left picture. The new $\Omega$-tropical series is $\min(2x,x + \frac{2}{15},y,1-x,1-y,\frac{1}{3})$ and the corresponding tropical curve is presented on the right. The fat point is $(\frac{1}{5},\frac{1}{2})$. Note that there appears a new face where $2x$ is the dominating monomial.}
\label{fig_3dpicture}
\end{figure}

In Lemma~\ref{lem_singlegp} we prove that each individual $G_\p$ simply contracts a face of a tropical curve $C(f)$ until $C(G_\p f)$ passes through $\p$, see Figure~\ref{fig_ShrinkPhi}. In Proposition~\ref{prop_gpsconvergence} we will prove that $G_P$ can be obtained as the limit of repetitive applications $G_\p$ for $\p\in P$. 

We denote by $0_\Omega$ the function $f\equiv 0$ on $\Omega$.

\begin{lemma}
For $\p\in\Omega^\circ$ we have $G_{\p} 0_\Omega (z) = \min(l_\Omega(z),l_\Omega(\p))$.
\end{lemma}
\begin{proof}
Indeed, all the coefficients, except $a_{00}$, in the canonical form of $G_{\p} 0_\Omega$ can not be less than in $l_{\Omega}$ by Remark~\ref{rem_lomegaestimate}, and if $a_{00}$ were less than $l_\Omega(\p)$, then the function would be smooth at $\p$.
\end{proof}

\begin{proposition} 
\laber{prop_upperbound}
For any $z\in\Omega$ and $P=\{\p_1,\dots,\p_n\}$ the following inequality holds
$$G_P0_\Omega \leq n\cdot l_\Omega(z).\vv$$
\end{proposition}
\begin{proof}
For each point $\p\in P$ we consider the function $(G_\p 0_\Omega)(z)=\min(l_\Omega(z),l_\Omega(\p))$, which is not smooth at $\p$ and $(G_\p 0_\Omega)|_{\partial\Omega}=0$. Finally, $$G_P0_\Omega \leq \sum_{\p\in P}G_\p 0_\Omega\leq n\cdot l_\Omega.$$
\end{proof}

\begin{lemma}
\laber{lemma_tropicalseries}
If $f$ is an $\Omega$-tropical series, then $G_P f$ is an $\Omega$-tropical series. \vv
\end{lemma}
\begin{proof}
\label{proof_tropicalseries}

Let $g\in V(\Omega,P,f)$, $z_0\in \Omega^\circ$ and $K\subset\Omega^\circ$ be a compact set such that $z_0\in K^\circ$. Denote by $\C>0$ the maximum of $g$ on $K$. Consider the set $\mathcal{M}$  of all $(i,j)\in\ZZ^2$ for which there exist $d\in\RR,(x_0,y_0)\in K$ such that $0\leq (xi+yj+d)|_{\Omega^\circ}, ix_0+jy_0+d\leq  \C.$ The set $\mathcal{M}$ is finite by Lemma~\ref{lemma_estimate}. Therefore, the restriction of any tropical series $g\in V(\Omega,P,f)$ to $K$ can be expressed as a tropical polynomial $\min_{(i,j)\in\mathcal{M}}(ix+jy+a_{ij}(g))$. In particular, if we denote by $a_{ij}$ the infimum of $a_{ij}(g)$ for all $g\in V(\Omega,P,f)$ then $$G_P f|_K=\min_{(i,j)\in\mathcal{M}}(ix+jy+a_{ij}),$$
so $G_Pf$ is a tropical series.
 
It follows from Proposition~\ref{prop_upperbound}, that $G_P f \leq f+ n\cdot l_\Omega$. Then, $l_\Omega|_{\partial\Omega}=0$ by Lemma~\ref{lemma_lomegaiszero}.  Therefore $G_P f|_{\partial\Omega}=0$ and, thus, Lemma~\ref{lemma_usualisomegatropical} concludes the proof that $G_Pf$ is an $\Omega$-tropical series.
\end{proof}

\begin{definition}
\laber{def_add}
For an $\Omega$-tropical series $f$ in the canonical form (see \eqref{eq_series}, Definition~\ref{def_canonicalseries}), $(k,l)\in\A$, and $c\geq 0$ and  we denote by $\Add_{kl}^c f$ the $\Omega$-tropical series $$(\Add_{kl}^c f) (x,y)=\min\left(a_{kl}+c+kx+ly,\min\limits_{\substack{{(i,j)\in\A} \\ {(i,j)\ne (k,l)}}}(a_{ij}+ix+jy)\right).$$		  
\end{definition}

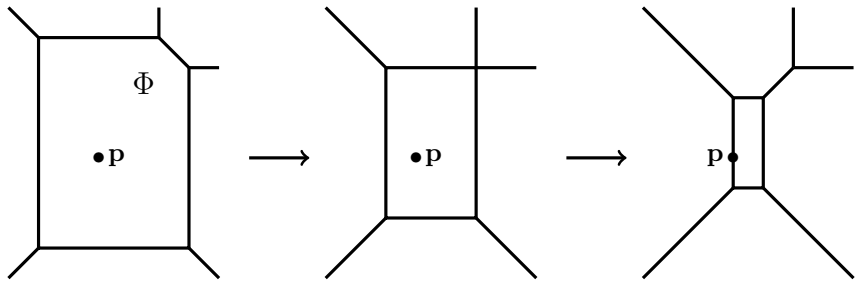
\begin{figure}[h]
    \centering

\begin{tikzpicture}[scale=0.4]
\draw[very thick](0,0)--++(1,1)--++(0,7)--++(4,0)--++(1,-1)--++(0,-6)--++(-5,0);
\draw[very thick](1,8)--++(-1,1);
\draw[very thick](5,8)--++(0,1);
\draw[very thick](6,7)--++(1,0);
\draw[very thick](6,1)--++(1,-1);
\draw(3,4)node{$\bullet$};
\draw(3,4)node[right]{$\p$};
\draw(4.5,6.5)node{\Large$\Phi$};
\draw[->][very thick](8,4)--(10,4);

\begin{scope}[xshift=300]
\draw[very thick](0,0)--++(2,2)--++(0,5)--++(3,0)--++(0,-5)--++(-3,0);
\draw[very thick](2,7)--++(-2,2);
\draw[very thick](5,7)--++(0,2);
\draw[very thick](5,7)--++(2,0);
\draw[very thick](5,2)--++(2,-2);
\draw(3,4)node{$\bullet$};
\draw(3,4)node[right]{$\p$};
\draw[->][very thick](8,4)--(10,4);
\end{scope}

\begin{scope}[xshift=600]
\draw[very thick](0,0)--++(3,3)--++(0,3)--++(1,0)--++(0,-3)--++(-1,0);
\draw[very thick](3,6)--++(-3,3);
\draw[very thick](5,7)--++(0,2);
\draw[very thick](5,7)--++(2,0);
\draw[very thick](5,7)--++(-1,-1);
\draw[very thick](4,3)--++(3,-3);
\draw(3,4)node{$\bullet$};
\draw(3,4)node[left]{$\p$};
\end{scope}

\end{tikzpicture}
\caption{Illustration for Remark~\ref{rem_smooth}. The operator $G_\p$ shrinks the face $\Phi$ where $\p$ belongs to. Firstly, $t=0$, then $t=0.5$, and finally $t=1$ in $\Add_{ij}^{ct}f$. Note that combinatorics of the curve can change when $t$ goes from $0$ to $1$.} 
\label{fig_ShrinkPhi}
\end{figure}

\begin{lemma}
\laber{lem_singlegp}
Let $f = \min_{(i,j)\in\A_\Omega}(ix+jy+a_{ij})$ be an $\Omega$-tropical series in the canonical form, suppose that $\p=(x_0,y_0)\in\Omega^\circ\setminus C(f)$. Suppose that $f$ is equal to $kx+ly+a_{kl}$ near $\p$.  
Consider the function 
\begin{equation}
\label{eq_gp}
f'(x,y) = \min_{(i,j)\in\A_\Omega, (i,j)\ne(k,l)}(ix+jy+a_{ij}).
\end{equation} Then, $G_\p f = \Add_{kl}^c f$ with $c=f'(\p)-kx_0-ly_0$.
\end{lemma}

\begin{proof} $G_\p(f)$ is at most $\min\left(f',kx+ly+(f'(\p)-kx_0-ly_0)\right)$ by definition. Therefore $f$ and $G_pf$ differ only at one monomial. Also, direct calculation shows that $\min(f',kx+ly+c)$ is smooth at $\p$ as long as $c< f'(\p)-kx_0-ly_0$, which finishes the proof.
\end{proof}

\begin{corollary}
In the notation of Definition~\ref{def_generalweighteddistance}, for a point $\p\in\Omega^\circ$, for each $z\in\Omega$ we have $$(G_\p 0_\Omega)(z)=\min\{l_\Omega(z),l_\Omega(\p)\}.$$ 
\end{corollary}

\begin{remark}
\laber{rem_smooth} Suppose that $G_\p f=\Add_{kl}^c f$.
We can include the operator $\Add_{kl}^c$  into a continuous family of operators $$f\to \Add_{kl}^{ct}f,\text{\ where $t\in[0,1]$}.$$ This allows us to observe the tropical curve {\it during} the application of $\Add_{kl}^c$, in other words, we look at the family of curves defined by tropical series $\Add_{kl}^{ct}f$ for $t\in[0,1]$. See Figure~\ref{fig_ShrinkPhi}. \end{remark}

\section{Dynamic generated by $G_\p$ for $\p\in P$.}
\label{sec_dynamic}

Recall that $P=\{\p_i\}_{i=1}^n, P\subset\Omega^\circ$.  Let $Q=\{\q_1,\q_2,\dots\}$ be an infinite sequence of points in $P$ where each point $\p_i,i=1,\dots, n$ appears infinite number of times. Let $f$ be any $\Omega$-tropical series. Consider a sequence of $\Omega$-tropical series $\{ f_m\}_{m=1}^\infty$ defined recursively as $$f_1=f, f_{m+1}=G_{{\bf q}_m} f_m.$$

\begin{proposition}\label{prop_gpsconvergence}
The sequence $\{ f_m \}_{m=1}^\infty$ uniformly converges to $G_P f$.
\end{proposition}

\begin{proof}First of all, $G_P f$ has an upper bound $f+nl_{\Omega}$ by arguments as in Proposition~\ref{prop_upperbound}. 
Applying Lemma \ref{lem_gpmponotone}, induction on $m$ and the obvious fact that $G_{\p_m}G_Pf=G_Pf$ we have that $f_m\leq G_P f$ for all $m.$ 
It follows from Lemmata~\ref{lemma_estimate},~\ref{lem_singlegp} that $G_{{\bf q}_m}, m=1,\dots$ change only a certain fixed finite subset of monomials in $f_m$. This implies the uniform convergence: since the family $\{f_{m}\}_{m=1}^\infty$ is pointwise monotone and bounded, it converges to some $\Omega$-tropical series $\tilde f\leq G_P f$. Indeed, to find the canonical form of $\tilde f$ we can take the limits (as $m\to\infty$) of the coefficients for $f_m$ in their canonical forms \eqref{eq_series}. 

It is clear that $\tilde f$ is not smooth at all the points $P$. Therefore, by definition of $G_P$ we have $\tilde f\geq G_Pf$, which finishes the proof.
\end{proof}

\begin{remark}
Note that in the case when $\Omega$ is a lattice polygon and the points $P$ are lattice points, all the increments $c$ of the coefficients in $G_\p=\Add_{kl}^c$ are integers, and therefore the sequence $\{f_m\}$ always stabilizes after a {\bf finite} number of steps.
\end{remark}

\begin{lemma}
\laber{lemma_ecloseseries}
Let $\e>0,\B\subset\ZZ^2$ and $f,g$ be two tropical series in $\Omega^\circ$ written as 
$$f(x,y)=\min_{(i,j)\in\B}(ix+jy+a_{ij}),g(x,y)=\min_{(i,j)\in\B}(ix+jy+a_{ij}+\delta_{ij}).$$
If $|\delta_{ij}|<\e$  for each $(i,j)\in\B$, then $C(f)$ is $2\e$-close to $C(g)$.
\end{lemma}
\begin{proof}

Let $z\in C(f)$, $l_1,l_2$ be two monomials of $f$, which are minimal at $z$. Suppose that $B_{2\e}(z)\cap C(g) = \varnothing$. Therefore $g|_{B_{2\e}(z)} = l(x,y)$ where $l:\RR^2\to R$ is a linear function with integer slope. Without loss of generality we may suppose that $z=(0,0), f(z)=0$ and $l(x,y) = c$. Clearly, $c\geq -\e$. At least one of $l_1,l_2$ is not a constant, by $SL(2,\ZZ)$-change of coordinates  we may suppose that $l_1 = x$. Then, in $g$, the monomial $x$ has the coefficient $\delta_{1,0}$ which satisfies $|\delta_{1,0}|\leq \e$. But then $x+\delta_{1,0}\leq -\e\leq c$ at a point in $B_{2\e}((0,0))$, so this point belongs to $C(g)$, which is a contradiction.
\end{proof}

\begin{remark}
\label{rem_curvesnearlimit}
Note that if $G_{{\bf q}_n}\dots G_{{\bf q}_1}f$ is close to the limit $G_Pf$, then by Lemma~\ref{lemma_ecloseseries} we see that the corresponding tropical curves are also close to each other.
\end{remark}

\begin{definition}
For two $\Omega$-tropical series $f=\inf(ix+jy+a_{ij}), (i,j)\in \B$ and $g = \inf(ix+jy+b_{ij}), (i,j)\in \B$ we define $\rho(f,g) = \sup_\B(|a_{ij}-b_{ij}|)$.
\end{definition}

\begin{lemma}
\label{lem_distance}
If $f,g$ are two $\Omega$-tropical series and $\p\in\Omega^\circ$, then $\rho(G_\p f,G_\p g)\leq \rho(f,g)$.
\end{lemma}
\begin{proof}
Note that for each $z\in \Omega$, $|f(z)-g(z)|\leq \rho(f,g)$. Therefore, if $\p$ belong to the face where $ix+jy+a_{ij} = f(x,y)$ and $ix+jy+b_{ij} = g(x,y)$, then it follows from \eqref{eq_gp} that the coefficients in monomial $ix+jy$ in $G_\p f,G_\p g$ differ by at most $\rho(f,g)$. 

Let $\p$ belong to different faces in $C(f), C(g)$, i.e. $ix+jy+a_{ij} = f(x,y), i'x+j'y+b_{i'j'} = g(x,y)$ near $\p$. Without loss of generality we may suppose that $i'=j'=0$ and $\p=(0,0)$. Therefore, $a_{ij}\leq a_{00}, b_{ij}\geq b_{00}, a_{00}\leq b_{00}+\rho(f,g)$. Finally, $G_\p f$ increases $a_{ij}$, clearly new $a_{ij}$ is at most $a_{00}\leq b_{00}+\rho(f,g)\leq b_{ij}+\rho(f,g)$. Other inequalities for the coefficients can be obtained similarly. 
\end{proof}

\section{A lift of a wave operator $G_\p$ in characteristic two}
\label{sec_lift}
Let $\mathbb K$ be a field with a valuation map $\mathrm{val}:\mathbb K^*\to \RR$. We use the convention $\mathrm{val}(a+b)\geq \min (\mathrm{val}(a)+\mathrm{val}(b)),\mathrm{val}(0)=+\infty$. To each polynomial $$F(X,Y)=\sum_{(i,j)\in\A}A_{ij}X^iY^j, A_{ij}\in \mathbb K^*$$ we associate the tropical polynomial $$\mathrm{Trop}(F)(x,y) = \min_{(i,j)\in\A}(\mathrm{val}(A_{ij})+ix+jy).$$

Historically, operators $G_\p$ appeared as continuous incarnations of waves in sandpiles, see \cite{us}. However, it is naturally to ask about their ``detropicalized'' version $S_\p, \p\in\mathbb (K^*)^2$, namely, how to lift $G_\p$ to the ring on polynomials (or series) over $\mathbb K$.

We managed to do that only in characteristic two. The formula is as follows: $$(S_\p F)(z)=F(z)+F(\sqrt{z\p})^2/F(\p) \text{ for \ } z\in\mathbb (K^*)^2,$$ if $F(\p)\ne 0$ and $S_\p F = F$ if $F(\p)=0$. We multiply the points coordinatewise. 

\begin{theorem}
For each $F\in\mathbb K[x,y]$ and $\p\in(\mathbb K^*)^2$ the following condition holds $$G_{\mathrm{val}(\p)} (\mathrm{Trop}(F))=\mathrm{Trop}(S_\p F).$$
\end{theorem}
\begin{remark}
It is easy to check that $(S_\p F)(\p)=0$, which implies that $C(\mathrm{Trop}(S_\p F))$ passes through $\p$. In turn it implies that $S_\p S_\p F= S_\p F$. 
\end{remark}

\begin{proof} 
Suppose that $F(x,y) = \sum A_{ij}X^iY^j$ and $\p=(p_1,p_2)$. Suppose that $A_{kl}X^kY^l$ is the only monomial with minimal valuation at $\p$ (i.e. $C(\mathrm{Trop}(F))$ does not pass through $\mathrm{val}(\p)$). 
 Then $$F(\sqrt{z\p})^2/f(\p) = \sum X^iY^j p_1^ip_2^j A_{ij}^2/f(\p).$$ 
 Note that 
 $$\mathrm{val}(A_{ij}^2p_1^ip_2^j/F(\p))=\mathrm{val}\Big(A_{ij}\frac{A_{ij}p_1^ip_2^j}{A_{kl}p_1^kp_2^l}\Big)>\mathrm{val}(A_{ij}),$$ 
 therefore the valuation of all coefficients for $ix+jy$ of $S_\p f$ and $f$ are the same except $kx+ly$. Presenting $F(z)$ near $\p$ as $F(x,y) = A_{kl}x^kx^l + G(x,y)$ we complute the new coefficient for $kx+ly$ as $$A_{kl}+A_{kl}\frac{A_{kl}p_1^kp_2^l}{F(\p)} = A_{kl}\frac{-G(\p)}{F(\p)},$$ and $\mathrm{val}(G(\p)/F(\p))$ coincides with the expression for $c$ in Lemma~\ref{lem_singlegp}. Note that if two or more valuations of monomials of $F$ are equal at $\p$, then $\mathrm{val}(G(\p)/F(\p)) = 0$ and no one coefficient of $\mathrm{Trop}(F)$ changes.
\end{proof}

Partial motivation to introduce the operators $S_\p$ was to prove the finiteness of the dynamic of $G_{\p_i}$. Some kind of stabilization (in the smallest terms) for $S_{\p_i}$ would imply the following finiteness property for $G_{\p_i}$.

\begin{question} Let $P=\{\p_1,\dots,p_n\}\subset \Omega^\circ$. Is it true that $G_P 0_\Omega = (\prod G_{\p_1}\dots G_{\p_n})^k 0_\Omega$ for a finite $k$?
\end{question}

%

\section{Contracting a face}

By a change of coordinates for a function $f:\RR^2\to\RR$ we mean $$f(x,y)\to f(ax+by+e,cx+dy+f)+nx+my+k$$ where $a,b,c,d,n,m\in\ZZ,e,f,k\in\RR, ad-bc=1$.

\begin{definition}
\laber{def_smoothvertex}
A vertex $V$ of a tropical curve $C(f)$ is {\it smooth} if the restriction of $f$ to a small neighborhood of $V$ can be presented as $\min(x,y,0)$ after a change of coordinates. A vertex $V$ of $C(f)$ is called {\it a node} if the restriction of $f$ to a small neighborhood of $V$ can be presented as $\min(x,y,0,x+y)$  after a change of coordinates. An edge of $C(f)$ has weight $m$ if the restriction of $f$ to a small neighborhood of any internal point in the edge is $\min(0,mx)$ after a change of coordinates. 
See Figure~\ref{fig_balancing} for examples of smooth and non-smooth vertices.
\end{definition}


%
%
%
%
%
%
%

\begin{remark}
At every vertex of a tropical curve the {\it balancing condition} is satisfied, i.e. the weighted sum of the outgoings primitive vectors in the directions of edges is zero, see Figure~\ref{fig_balancing}.
\end{remark}

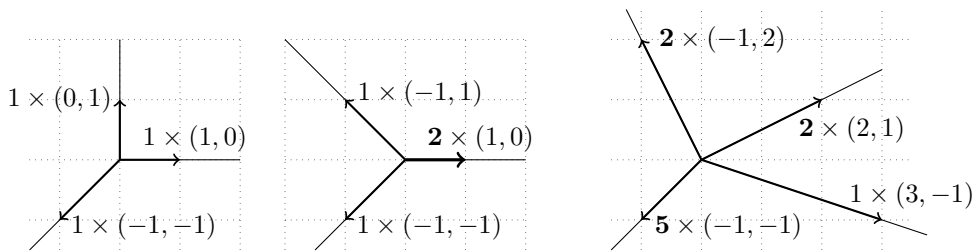
\begin{figure}[h]
    \centering

\begin{tikzpicture}[scale=0.4]
\draw(0,0)--++(4,0);
\draw(0,0)--++(0,4);
\draw(0,0)--++(-3,-3);
\draw[->][thick](0,0)--++(2,0);
\draw[->][thick](0,0)--++(0,2);
\draw[->][thick](0,0)--++(-2,-2);
\draw(2.5,0.7)node{$1\times (1,0)$};
\draw(-2,2)node{$1\times (0,1)$};
\draw(0.8,-2.2)node{$1\times (-1,-1)$};
\draw[black,step = 2.0,very thin, dotted](-3,-3) grid (4,4);

\begin{scope}[xshift=270]
\draw(0,0)--++(-4,4);
\draw(0,0)--++(-3,-3);
\draw(0,0)--++(4,0);
\draw[->][very thick](0,0)--++(2,0);
\draw[->][thick](0,0)--++(-2,2);
\draw[->][thick](0,0)--++(-2,-2);
\draw(2.5,0.7)node{${\bf 2}\times (1,0)$};
\draw(0.5,2.2)node{$1\times (-1,1)$};
\draw(0.8,-2.2)node{$1\times (-1,-1)$};
\draw[black,step = 2.0,very thin, dotted](-4,-3) grid (4,4);
\end{scope}

\begin{scope}[xshift=550]
\draw(0,0)--++(-3,-3);
\draw(0,0)--++(-2.5,5);
\draw(0,0)--++(6,3);
\draw(0,0)--++(7.5,-2.5);
\draw[->][thick](0,0)--++(-2,-2);
\draw[->][thick](0,0)--++(-2,4);
\draw[->][thick](0,0)--++(4,2);
\draw[->][thick](0,0)--++(6,-2);
\draw(0.9,-2.2)node{${\bf 5}\times (-1,-1)$};
\draw(7,-1.2)node{$1\times (3,-1)$};
\draw(5,1)node{${\bf 2}\times (2,1)$};
\draw(0.7,4)node{${\bf 2}\times (-1,2)$};
\draw[black,step = 2.0,very thin, dotted](-3,-3) grid (7,5);
\end{scope}

\end{tikzpicture}
\caption{Examples of balancing condition in local pictures of tropical curves near vertices. The notation ${\bf m}\times (p,q)$ means that the corresponding edge has the weight $m$ and the primitive vector $(p,q).$ The vertex on the left picture is smooth, the vertices in the middle and right pictures are neither smooth nor nodal.}
\label{fig_balancing}

\end{figure}

\begin{definition}
\laber{def_smoothcorner}
A corner of a $\QQ$-polygon $\Delta$ is called {\it unimodular} if the primitive vectors of the directions of the edges of $\Delta$ at this corner give a $\ZZ$-basis of $\ZZ^2$. A $\QQ$-polygon is unimodular if all its corners are unimodular.
\end{definition}

\begin{definition}
\laber{def_nodal}
A $\Delta$-tropical curve is called {\it smooth or nodal} if all its vertices in $\Delta^\circ$ are smooth or nodal (see Definition~\ref{def_smoothvertex}).\m{note that smooth is only in the interior. we don't care about corners because smoothness there follows from niceness} In particular, this curve has no edges of weight bigger than one.
\end{definition}

Let a $\Delta$-tropical polynomial $f$ define a tropical curve $C(f)\subset \Delta$. Let $\p$ belong to the interior of a face $\Phi$ of the complement $\Delta\setminus C(f)$ of $C(f)$. Suppose that all corners of $\Phi$ are unimodular. We can find $c>0$ and $(i,j)\in\ZZ^2$ such that $$G_\p f=\Add_{ij}^c f.$$ Consider the family $\{\Add_{ij}^{ct} f\}_{t\in[0,1]}$ of tropical polynomials (see Remark~\ref{rem_smooth}). Denote by $\Phi^t$ the face of $\Add_{ij}^{ct} f$ to which $\p$ belongs.

Consider a side $S$ of the face $\Phi$ and two other sides $S_1$ and $S_2$ of $\Phi$ which are the neighbors of $S$. Applying $SL(2,\ZZ)$-change of coordinates and homothety we may suppose that $S$ is the interval with endpoints $(0,0),(1,0)$. We may assume, then, that the neighborhood of $S$ is locally coincide with $C(\tilde f)$, where 
$$\tilde f(x,y)=\min(0,y, x+n_1y+c_1, -x+n_2y+c_2), n_1,n_2\in\ZZ, c_1,c_2\in \RR,$$ because both endpoints of $S$ are smooth vertices of $C(f)$. Since the endpoints of $S$ are $(0,0),(1,0)$, we see that $c_1=0,c_2=1$. We suppose that $\Phi$ is the face where the function $0+0x+0y$ is the least monomial in $\tilde f$.

The curve $C(\Add_{0,0}^{ct}\tilde f)$ in the neighborhood of $S$ is given by the tropical polynomial $$\tilde f_t(x,y)=\min(ct,y, x+n_1y, -x+n_2y+1).$$
For small $t>0$ denote by $S^t$ the side of $\Phi^t$ (recall that $\Phi^t$ is a face of the curve $C(\Add_{0,0}^{ct}\tilde f)$) which is close and parallel to the side $S$  of the face $\Phi$.  It is easy to find the coordinates of the vertices of $S^t$ by direct calculation: they are $(ct(1-n_1),ct)$ and$(ct(n_2-1)+1,ct)$. The length of $S^t$ is therefore $ct(n_2-1)+1 - ct(1-n_1) = 1 + ct(n_1+n_2-2)$. We just proved the following lemma.

\begin{lemma}
\label{lemma_sideshorter}
In the above notation, two facts are equivalent:
\begin{itemize}
\item $S^t$ is shorter then $S$ for small $t>0$,
\item $n_1+n_2<2$.
\end{itemize}
\end{lemma}
 
\begin{corollary}
\laber{cor_smooth}
For the above situation there are three cases:
\begin{enumerate}[a)]
\item $n_1+n_2< 0$, this corresponds to collapsing the face $\Phi$ to $\p$ as $t\to 1$,
\item $n_1+n_2=0$, corresponds to collapsing the face $\Phi$ to a (possibly degenerate) interval containing $\p$ as $t\to 1$,
\item $n_1+n_2=1$, note that in this case $(1,n_1)+(-1,n_2)=(0,1)$.
\end{enumerate}
\end{corollary} 

\begin{definition}
\label{def_perestroika}
We say that a continuous family of tropical curves has a {\it nodal perestroika} (see Figure~\ref{fig_ShrinkPhi}) if all the curves, except one, are smooth, and non-smooth curve has only one nodal point, and the family near it is given by $\min(x,y,t,x+y)$ for $t\in [-\e,\e]$ up to $SL(2,\ZZ)$-change of coordinates.
\end{definition}

\begin{lemma}
\laber{lemma_smooth}
If all corners of $\Phi$ are unimodular, $\p\in \Phi$, $G_\p f=\Add_{ij}^c f$, then all the vertices of $\Phi^t, t\in[0,1)$  are smooth or nodal vertices of the curve $C(\Add_{ij}^{ct} f)$. If $\Phi$ is not contracted to a point or an interval by applying $G_\p$ to $f$, then the vertices of $\Phi^1$ are smooth or nodal as well.
\end{lemma}

\begin{proof}
\label{proof_lemmasmooth}
The combinatorial type of $\Phi_t$ can only change when at least one of the sides of the $\Phi^t$ is getting shrinked to a point for some $t$. Choose the minimal such $t=t_0$, and denote one of the shrinking sides by $S$. Corollary~\ref{cor_smooth} tells us that cases a), b) correspond to collapsing the face, so $t_0=1$, hence in these cases the lemma is proven.   

We assume that $t_0<1$ and the case c) in Corollary~\ref{cor_smooth} takes place. 

If neither $S_1$ nor $S_2$ gets contracted when we pass from $C(f)$ to $C(\Add_{ij}^{ct_0} f)$, then we see a {\it nodal perestroika} (Definition~\ref{def_perestroika}).
If $S_2$ is contracted by passing from $C(f)$ to $C(\Add_{ij}^{ct_0} f)$, then the direct computation using Corollary~\ref{cor_smooth} c) implies that the side $S_3$ of $\Phi$, which is next after $S_2$, is parallel to $S_2$ and therefore the whole face $\Phi$ is contracted by $\Add_{ij}^{ct_0}$ which is a contradiction. The case when $S_1$ is contracted is handled by the same argument.
\end{proof}

\begin{corollary}
The edges of $C(\Add_{ij}^{ct} f)\cap\Phi$ for $0\leq t<1$ have weight $1.$  
\end{corollary}

\section{$\QQ$-polygons}
\label{sec_qpolygons}
\begin{definition}
\label{def_qpolygon}
Let $\Delta\subset\RR^2$ be a finite intersection of half-planes (at least one) with rational slopes.
We call $\Delta$ a {\it $\QQ$-polygon} if it is a closed set with non-empty interior.
\end{definition}

\begin{definition}
\laber{def_minimalform}
We say that a tropical series $f$ on $\Omega$ is presented in the {\it small canonical form} if $f$ is written as 
\begin{equation}
f(x,y) = \min_{(i,j)\in\B_f}(ix+jy+a_{ij})
\end{equation}
where all $a_{ij}$ are taken from the canonical form and $\B_f\subset \A_\Omega$ consists of monomials $ix+jy+a_{ij}$ which are equal to $f$ at at least one point in $\Omega^\circ$.
\end{definition}

\begin{example}
The small canonical form for Example~\ref{ex_bigform} is $\min(x,y,1-x,1-y,1/3)$.
\end{example}

\begin{remark}
\laber{rem_smallcanonicalGp}
Note that for a $\QQ$-polygon $\Delta$, the small canonical form of the function $l_{\Delta}$ is a $\Delta$-tropical {\bf polynomial}, i.e. it has only finite number of monomials. It follows from the estimate in Proposition~\ref{prop_upperbound} that  the small canonical form of $G_P f$ is a $\Delta$-tropical polynomial too, for all $\Delta$-tropical polynomials $f$.
\end{remark}

Let us fix a $\QQ$-polygon $\Delta$. Consider a $\Delta$-tropical polynomial $f$  in the small canonical form (Definition~\ref{def_minimalform}). Let us analyze the behavior of $f$ near the boundary. 

In the neighborhood of each side $S$ of $\Delta$ the function $f$ can be
locally written as $(x,y)\mapsto ix+jy+a_{ij},$ where $(i,j)\in \B_f$ 
and the vector $(i,j)$ is orthogonal to $S$. This integer vector $(i,j)$ is a multiple of a
certain primitive vector, i.e. $(i,j)=m_f(S)n(S),$ where $n(S)$ is {\it the inward
  primitive normal} vector to $S$ of $\Delta$ and $m_f(S)\in\ZZ_{>0}$ is a number.   Thus, we constructed the function $m_f$ on the set $S(\Delta)$ of the
sides of $\Delta$, $m_f\colon S(\Delta)\rightarrow \mathbb{Z}_{>0}$.
\begin{definition}
\laber{def_qdegree} The aforementioned function
  $m_f$ is called the {\it quasi-degree} for the  $\Delta$-tropical curve $C$.
\end{definition}
\begin{remark}
Note that $m_f(S)n(S)\in \B_f$ for each $S\in S(\Delta).$ The convex hull of the set $$\{ m_f(S)n(S)\}_{S\in S(\Delta)}$$
contains $\B_f$, since the monomials from the outside of this
convex hull can not contribute to $f|_\Delta$. 
\end{remark}

\begin{definition}
\laber{def_nicedegree}
A quasi-degree $m_f$ is called {\it nice} if for each side $S\in S(\Delta)$ with $m_f(S)>1$ we have $m_f(S_1)=m_f(S_2)=1$ for the neighboring sides $S_1,S_2$ of $S$.
\end{definition}

\begin{theorem}
\label{th_verge}
Let $\Delta$ be a unimodular $\QQ$-polygon (Definition~\ref{def_smoothcorner}). Suppose that a quasi-degree $d$ on $\Delta$ is nice. Then for any $\e>0$ there exists a tropical $\Delta$ polynomial $g$ such that $m_g=d$, the curve $C(g)$ is smooth (Definition~\ref{def_nodal}) and is contained in the $\e$-neighborhood of $\partial \Delta$.
\end{theorem}

\begin{proof}
Let $\{S_k\}_{k=1}^n$ be the sides of $\Delta$. Suppose that each side $S_k$ is given by $i_kx+j_ky+a_k=0$ and all these linear functions are non-negative on $\Delta$.  Choose small $\delta>0$. For each $k=1,\dots, n$ we consider the following tropical polynomial: 
$$f_k(x,y) = \min_{l=1,\dots,d(S_k)}\left(l(i_kx+j_ky+a_k)-\frac{l(l+1)}{2d(S_k)}\delta\right).$$
The tropical curve defined by $f_k$ is the collection of $d(S_k)-1$ lines parallel to $S_k$ with distance $\delta$ between them. Define $g$ as 

$$g(x,y) = \min\big(\e/2, \min_{k=1,\dots, n} f_k(x,y)\big).$$
Clearly, $C(g)\cap \Delta$ is contained in the $\e$-neighborhood of $\partial\Delta$. It is a local calculation near each corner that $C(g)\subset \RR^2$ is a smooth tropical curve: since the quasi-degree is nice, so near a corner of $\Delta$, $C(g)$ is given locally by $$\min(\e/2,x,y, 2y-\frac{1}{n}, 3y-\frac{3}{n}, 4y-\frac{6}{n},\dots)$$ where $n=\frac{d(S_k)}{\delta}$. Such a curve has an edge locally given by $x=\e/2$ and, if $\delta$ is small enough, $d(S_k)-1$ edges locally given by $y=\frac{k}{n},1\leq k\leq d(S_k)$, and these edges meet in smooth position, see Figure~\ref{fig_g} for an illustration. 

\end{proof}

\begin{figure}
\begin{tikzpicture}
\begin{scope}[scale=0.6]
\draw[dashed] (0,5)--(0,0)--(5,0);
\draw (0,0)--(0.33,1.33)--(0.33,5);
\draw (0.33,1.33)--(0.66,2.33)--(0.66,5);
\draw (0.66,2.33)--(1,3)--(1,5);
\draw (1,3)--(1.7,3.7)--(1.7,5);
\draw (1.7,3.7)--(5,3.7);
\end{scope}
\begin{scope}[scale=0.15, xshift=60cm, rotate=90]
\draw(2,0)--++(0,4)--++(18,36)--++(0,-38)--++(-2,-2)--++(-16,0);
\draw[fill=black](0,0)--++(2,0)--++(0,4)--++(-2,-4);
\draw[fill=black](20,0)--++(-2,0)--++(2,2)--++(0,-2);
\draw(2,0)--++(2,1)--++(1,1)--++(0,5)--++(-3,-3);
\draw(4,1)--++(13,0)--++(1,-1);
\draw(5,2)--++(12,0)--++(0,-1);
\draw(20,40)--++(-1,-4)--++(-1,-3)--++(-13,-26);
\draw(18,33)--++(0,-30)--++(-1,-1);
\draw(19,36)--++(0,-33)--++(-1,0);
\draw(20,2)--++(-1,1);
\end{scope}
\end{tikzpicture}
\caption{Left: the curve corresponding to the function $g$ from Theorem~\ref{th_verge}, near a corner, $d(S_k)=4$. Each vertex $V$ of the curve is smooth because $g$ is locally presented as $\min(y,kx,(k+1)x)$ near $V$. Right: an example of $C(g)$ for $g$ in Lemma~\ref{lemma_nicest}. Colored corners symbolize that  a quasidegree was not nice, and we made blow-ups at these corners.}
\label{fig_g}
\end{figure}
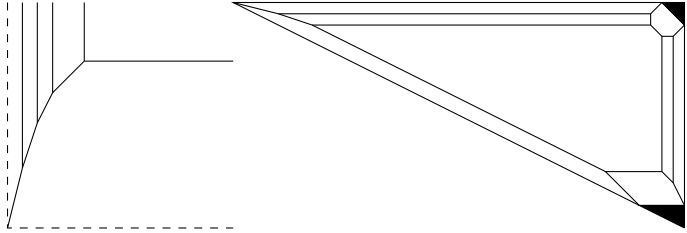

\section{Exhausting polygons}
\label{sec_exhausting}

\begin{definition}
\laber{def_fOmegaP}
Let $\p_1,\dots \p_n\in \Omega^\circ$ be different points, $P=\{\p_1,\dots,\p_n\}$. We denote by $f_{\Omega,P}$ the pointwise minimum among all $\Omega$-tropical
series non-smooth at all the points $\p_1,\dots,\p_n$.
\end{definition}

\begin{lemma}
\laber{lemma_levelsets}
If $\Omega$ is bounded, then for any $\e>0$ the set $\Omega_\e=\{x\in\Omega|f_{\Omega,P}\geq\e\}$ is a $\QQ$-polygon and $f_{\Omega,P}|_{\Omega_\e}$ is a tropical polynomial. 
\end{lemma}
\begin{proof}
Note that $G_P 0_\Omega=f_{\Omega,P}(x)$ by the definition of the latter, so it follows from Lemma~\ref{lemma_tropicalseries} that $f_{\Omega,P}$ is continuous and vanishes at $\partial\Omega$. Since $\Omega$ is bounded, the set  $f_{\Omega,P}=\e$ is a curve disjoint
from $\partial\Omega.$ We claim that the intersection of $\Omega_\e$ with $C(f_{\Omega,P})$ is a graph with a finite number of vertices. \m{a tropical series on level sets is just a tropical polynomial} Suppose the contrary. Then a sequence of vertices of this graph converges to a point $z\in\Omega^\circ.$ Thus, there is no neighborhood of $z$ where the
series $f_{\Omega,P}$ can be represented by a tropical
polynomial, which is a contradiction with Definition~\ref{def_tropseries}.  The finiteness of the number of vertices implies that there is only a finite number of monomials participating in the restriction of $f_{\Omega,P}$ to the domain $\Omega_\e,$ therefore the restriction is a tropical polynomial.
\end{proof}


\begin{lemma}
\laber{lemma_levelset2}
In the above hypothesis, we extend $f_{\Omega_\e,P}$ to $\Omega$ using the presentation of $f_{\Omega_\e,P}$ in the small canonical form (Definition~\ref{def_minimalform}). 
In the hypothesis of the previous lemma, if
$f_{\Omega,P}(\p)\geq\e$ for each $\p\in P$, then  we have $f_{\Omega,P}=f_{\Omega_\e,P}+\e$ on $\Omega_\e$. Also $f_{\Omega_\e,P}+\e\geq f_{\Omega,P}$ on $\Omega$.
\end{lemma}
\begin{proof}
On $\Omega_\e$ we have that $f_{\Omega,P}-\e\geq f_{\Omega_\e,P}$ by the definition of the latter. Then, two functions $f_{\Omega_\e,P}+\e, f_{\Omega,P}$ are equal on $\partial\Omega_\e$ and by the previous line the quasi-degree of $f_{\Omega_\e,P}$ is at most the quasi-degree of $(f_{\Omega,P}-\e)|_{\Omega_\e}$. Hence $f_{\Omega,P}$ can not decrease slowly than $f_{\Omega_\e,P}$ when we move from $\partial\Omega_\e$ towards $\partial\Omega$. Therefore $f_{\Omega_\e,P}+\e\geq f_{\Omega,P}$ on $\Omega\setminus\Omega_\e$. Since $f_{\Omega_\e,P}+\e\geq 0$ on $\Omega$ we obtain the estimate $f_{\Omega_\e,P}+\e\geq f_{\Omega,P}$ on $\Omega$ which concludes the proof. 
\end{proof}

Note that a $\QQ$-polygon is not necessary compact. It is easy to verify that a $\QQ$-polygon is admissible (Definition~\ref{def_omegaadmissible}). The next lemma provides us with a family of compact $\QQ$-polygons exhausting $\Omega$.

\begin{lemma}
\laber{lemma_epsfun} For any compact \m{we need K: for big slopes we have s lot of of trash near boundary, but Q-polygon is ok (finite number of slopes)} set $K\subset \Omega^\circ$ such that $P\subset K$ and for any $\e>0$ small enough there exists a $\QQ$-polygon $\Omega_{\e,K}\subset \Omega$ such that $B_{3\e}(K)\subset \Omega_{\e,K}$ and the following holds: 
$$f_{\Omega,P}= f_{\Omega_{\e,K},P}+\e \text{\ on\ } B_{3\e}(K).$$ 
\end{lemma}

%
%
%


\begin{proof}

Note that if $\Omega'\subset \Omega$, then $f_{\Omega',P}\leq f_{\Omega,P}$ automatically. We list several possible cases. A) $\Omega$ is a compact set, see Lemma~\ref{lemma_levelsets}. If $\Omega$ is not compact, then it is possible that B) $\Omega$ is a half-plane with the boundary of rational slope. Otherwise, $\partial\Omega$ has two asymptotes: C) of rational slope, D) of irrational slope, E) one of asymptotes is of rational slope and another is not.

Let $M = \max_K f_{\Omega,P}$. It follows from Lemma~\ref{lemma_estimate} that the set $I$ of monomials $(i,j)\in \ZZ^2$ such that there exists $a_{ij}$ such that $(a_{ij}+ix+jy)|_K\geq 0$ and $a_{ij}+ix+jy<M$ at a point of $K$ is finite. Only these monomials may contribute to $f_{\Omega,P}|_K$ and we are going to study their coefficients.

 For each $(i,j)\in I$ there are three possibilities:
i) for some $c$ the line $\{c+ix+jy=0\}$ is an asymptote of $\partial\Omega$; ii) for a compact set $K'$, containing $P$ and big enough, $ix+jy-\min_{K'\cap\Omega}(ix+jy)\geq 0$ on $\Omega$; iii) for a compact set $K'$, containing $P$ and big enough, $ix+jy-\min_{K'\cap\Omega}(ix+jy)> M$ on $K$.

In the case B) this implies that for $K'$ big enough the only monomials which contribute to $f_{K'\cap \Omega}$ are the multiples of the monomial giving $\partial\Omega$, therefore $f_{\Omega,P} = f_{\Omega\cap K', P}$, which reduces the proof to A).

The case D) is handled similarly: ii) is not possible, but i) implies that $f_{\Omega,P}\leq f_{\Omega\cap K',P}$ on $K$, which proves that $f_{\Omega,P}= f_{\Omega\cap K',P}$ on $K$.

In the case C) we prove that $\{f_{\Omega,P}\geq \e\}$ is a $\QQ$-polygon for $\e>0$ small enough. Indeed, let one of the asymptotes is given by $L=\{kx+ly+a_{kl} = 0\}$. Then, for the points $z\in L$ far enough from $P$ we have that the distance between $z$ and $\partial\Omega$ is at least $\e/2$ and therefore by Lemma~\ref{lemma_estimate}  the set $I'$ of monomials $a_{ij}+ix+jy$ which are non-negative on $\Omega$ and less than $\e$ at $z$ is finite. Therefore, by taking $K'$ big enough and containing all the points of intersection of the support lines to $\partial\Omega$ with directions in $I'$ (if there is no a point of intersection, it means that this is another asymptote and this is handled easily),  we may assume that $f_{\Omega,P} = \e$ is given by $L$ far enough from $P$ and the same for another asymptote. Therefore the curve $\{f_{\Omega}=\e\}$ is a $\QQ$-polygon.

The last case, E) is handled as follows: we take $K'$ as above and then find a line $L(x,y)=0$ with a rational slope close to the irrational slope of an asymptote, such that $L|_{K'\cap \Omega}\geq 0$ and we reduce the case to D) by considering $\Omega\cap \{L\geq 0\}$ instead of $\Omega$.  
\end{proof}

\begin{corollary}
\label{cor_epscurve}
Lemma~\ref{lemma_epsfun} implies that for $\e>0$ small enough the tropical
curves defined by $f_{\Omega,P} $ and $f_{\Omega_{\e,K},P}$ coincide on $K$,
i.e.  $$C(f_{\Omega,P})\cap K=C(f_{\Omega_{\e,K},P})\cap K.$$
\end{corollary}

\section{How to blow-up corners of a polygon}
\label{sec_blowup}
Let $p_1,p_2,q_1,q_2\in\ZZ$ such that $p_1q_2-p_2q_1\ne 0$ and let 
\begin{equation}
\label{eq_lambda}
\Lambda=\{(x,y)\in\RR^2|xp_1+yq_1\geq 0, xp_2+yq_2\geq 0\}.
\end{equation}

\begin{lemma}The set $\A_\Lambda$ (Definition~\ref{def_singlesupport}) is equal to the set
$$(\RR_{\geq 0}(p_1,q_1) \oplus \RR_{\geq 0}(p_2,q_2))\cap\ZZ^2.$$
\end{lemma}

\begin{proof}
Any vector $(p,q)\in\ZZ^2$ can be written as $(p,q)=\alpha\cdot (p_1,q_1)+\beta\cdot (p_2,q_2)$ with $\alpha,\beta\in\RR$. Then, if $\alpha<0$ or $\beta<0$, then $px+qy$ is negative on one side of $\Lambda$. 
\end{proof}



\begin{definition}
\laber{def_eblowup}
Suppose that $\Delta$ is a $\QQ$-polygon and $O=(0,0)$ is its vertex. Let $\e>0$. Let $\Lambda$ be as in \eqref{eq_lambda} such that $\Delta\subset\Lambda$ and $\Delta,\Lambda$ coincide in a neighborhood of $O$. We say that $\Delta'\to\Delta$ is the $\e$-blowup of $\Delta$ in a direction $(i,j)\in \A_\Lambda$ if $$\Delta'=\{(x,y)\in\Delta | ix+jy-\e\geq 0\}.$$
We say that this blow-up is made {\it with respect to} the lattice point $(i,j)$. Note that $\Delta'\subset \Delta$. We say that $\partial\Delta'\setminus\partial\Delta$ (i.e. the new side of $\Delta'$ obtained as cutting the corner at $O$) is the side, {\it dual} to the vector $(i,j)$.
\end{definition}

Note that we do not require that $(i,j)$ is a primitive vector. This will be important in Lemma~\ref{lemma_nicest}.

\begin{remark}
Note that if $\Lambda$ is unimodular (Definition~\ref{def_smoothcorner}) then $p_1q_2-p_2q_1= 1$ and there exists a preferred direction $(p_1+p_2,q_1+q_2)$ to perform a blow-up which produces two unimodular corners near the vertex of $\Lambda$.
\end{remark}

Let $f$ be any $\Lambda$-tropical polynomial written in the small canonical form (Definition~\ref{def_minimalform}). So, $\supp(f)\subset \A_{\Lambda}$ and is finite. Recall that $O=(0,0)$ is the corner of $\Lambda$. 
\begin{lemma}
\laber{lemma_makingnice} 
Consider any $\e>0$ small enough. There exist $\delta>0, N>0$ such that if $$(p,q)\in(\RR_{>0}(p_1,q_1) \oplus \RR_{>0}(p_2,q_2))\cap \ZZ^2,\sqrt{p^2+q^2}>N,$$ then $px+qy - \delta>f$ on $\Lambda\setminus B_\e(O)$. 
\end{lemma}

\begin{proof}
We consider the case $(p_1,q_1)=(1,0), (p_2,q_2)=(0,1)$, the general case can be handled in the same way.  If $\e$ is small enough, then we have $$f|_{B_\e(O)\cap\Lambda}=\min_{(p_i,q_i)\in\A} (p_ix+q_iy)$$ where $\A\subset\A_\Lambda$. It is enough to prove the statement for $(p,q)=(1,N)$, i.e. that if $N$ is big enough and $\delta>0$ is small enough, then $$x+Ny-\delta>\min_{(p_i,q_i)\in\A} (p_ix+q_iy) \text{\ for\ } (x,y)\in \Lambda\setminus B_\e(O).$$  

The cone $\Lambda$ is dissected on regions where each of $p_ix+q_iy$ is the minimal monomial. All these sectors except one satisfy $y>cx$ for a constant $c$ depending on $p_i,q_i$.  Therefore if $N$ is big enough then $x+Ny>(p_i+1)x+(q_i+1)y>p_ix+q_iy+\delta$ if $x$ or $y$ is bigger than $\delta$. The only region where we do not have the estimate $y>cx$ is the region where the minimal monomial $p_ix+q_iy$ satisfies $p_i=0$. In this region, again, $x+Ny>q_iy+\delta$ if $x$ or $y$ is bigger than $\delta$ and $N$ is big enough. 

\end{proof}

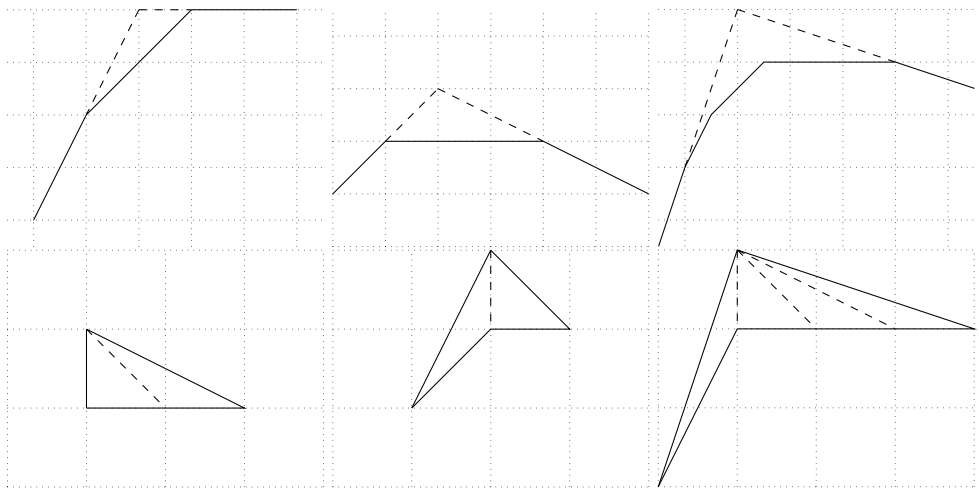
\begin{figure}
\begin{tikzpicture}[scale=0.35]
\draw[black,step = 2.0, very thin, dotted](-3,-9) grid (9,0);
\draw[dashed](2,0)--++(-2,-4);
\draw(0,-4)--++(-2,-4);
\draw[dashed](2,0)--++(2,0);
\draw(4,0)--++(4,0);
\draw(4,0)--++(-4,-4);
\end{tikzpicture}
\begin{tikzpicture}[scale=0.35]
\begin{scope}[xshift=400]
\draw[black,step = 2.0,very thin, dotted](-4,-10) grid (8,-1);
\draw(-4,-8)--++(2,2);
\draw[dashed](-2,-6)--++(2,2);
\draw[dashed](0,-4)--++(4,-2);
\draw(4,-6)--++(4,-2);
\draw(4,-6)--++(-6,0);
\end{scope}
\end{tikzpicture}
\begin{tikzpicture}[scale=0.35]
\begin{scope}[xshift=800]
\draw[black,step = 2.0,very thin, dotted](-3,-9) grid (9,0);
\draw[dashed](0,0)--++(-2,-6);
\draw[dashed](0,0)--++(6,-2);
\draw(-2,-6)--++(-1,-3);
\draw(6,-2)--++(3,-1);
\draw(-2,-6)--++(1,2)--++(2,2)--++(5,0);
\end{scope}
\end{tikzpicture}
\begin{tikzpicture}[scale=0.35]
\begin{scope}[yshift=-300]
\draw(0,-3)--++(0,-3);
\draw(0,-3)--++(6,-3)--++(-6,0);
\draw[black,step = 3.0,very thin, dotted](-3,-9) grid (9,0);
\draw[dashed](0,-3)--++(3,-3);
\end{scope}
\end{tikzpicture}
\begin{tikzpicture}[scale=0.35]
\begin{scope}[yshift=-300]
\draw(3,0)--++(-3,-6);
\draw(3,0)--++(3,-3)--++(-3,0)--++(-3,-3);
\draw[black,step = 3.0,very thin, dotted](-3,-9) grid (9,0);
\draw[dashed](3,0)--++(0,-3);
\end{scope}
\end{tikzpicture}
\begin{tikzpicture}[scale=0.35]
\begin{scope}[yshift=-300]
\draw(0,0)--++(-3,-9);
\draw(0,0)--++(9,-3)--++(-9,0)--++(-3,-6);
\draw[black,step = 3.0,very thin, dotted](-3,-9) grid (9,0);
\draw[dashed](0,0)--++(0,-3);
\draw[dashed](0,0)--++(3,-3);
\draw[dashed](0,0)--++(6,-3);
\end{scope}

\end{tikzpicture}

\caption{Above pictures show  non-unimodular corners $\Lambda$ (dashed lines). The corresponding below pictures present lattice points with respect to whom we should perform the blow-ups in Lemma~\ref{lemma_nicest},  in order to make all the corners unimodular: the result is shown by continuous lines above.  Dashed lines below show vectors dual to the new sides.}
\label{fig_blowup}
\end{figure}

\section{Nice tropical series}
\label{sec_nice}
\begin{definition}
\laber{def_nice}
Let $f$ be a $\Delta$-tropical series. We say that $f$ is {\it nice} if all the corners of $\Delta$ are unimodular (Definition~\ref{def_smoothcorner}) and the quasi-degree (Definition~\ref{def_qdegree}) $m_f$ is nice (Definition~\ref{def_nicedegree}). 
\end{definition}

\begin{lemma}
\laber{lemma_notnice}
Let $\Delta$ be a $\QQ$-polygon. Suppose that $f$ is a nice $\Delta$-tropical series. Then, $C(f)$ has exactly one edge of weight one passing through each corner of $\Delta$. 
\end{lemma}
\begin{proof} Suppose the contrary.
Applying $SL(2,\ZZ)$ transformation and translation we may assume that the corner in consideration is at the point $(0,0)$ and two neighboring vertices of $\Delta$ are the points $(0,a)$ and $(b,0)$. Denote these neighboring sides by $S_1,S_2$. Suppose that $m_f(S_1)=1,m_f(S_2)=k$. Then $f$ is given by $f'(x,y)=\min (y,kx)$ in a neighborhood of $(0,0)$, and the tropical edge defined by $f'$ has weight one.
\end{proof}

\begin{lemma}
\laber{lemma_nicest}
Let $\Lambda$ be a corner, as in \eqref{eq_lambda}, and $f$ be any $\Lambda$-tropical polynomial. Let $\e>0$ be any small number. There exists a finite sequence of blowups (Definition~\ref{def_eblowup}) $$\Lambda^n\to\dots\to\Lambda^3\to\Lambda^2\to\Lambda^1\to\Lambda,$$  and a nice (Definition~\ref{def_nice}) $\Lambda^n$-tropical polynomial $\tilde f$ on $\Lambda^n$ such that $f=\tilde f$ on $\Lambda\setminus B_\e(O)$. 
\end{lemma}

\begin{proof}
Consider any ordering $\{(i_k,j_k)\}_{k=1}^\infty$ of primitive vectors in $\A_{\Lambda}\backslash\{(p_1,q_1),(p_2,q_2)\}$ such that $i_{k+1}^2+j_{k+1}^2\geq i_k^2+j_k^2$ for any pair of consecutive (with respect to this order) primitive vectors. Choose $\delta>0$ small enough and denote by $\Lambda^k$ the $\delta$-blow-up of $\Lambda^{k-1}$ with respect to the vector $n_k(i_k,j_k)$ where $n_k\in\NN$ is chosen in such a way that $||n_k(i_k,j_k)||\geq N$ (see  Lemma~\ref{lemma_makingnice}).

Note that $\Lambda^{k-1}$ contains $k$ corners but only one of them can be blow-upped using the direction $(i_k,j_k)$; so there is no ambiguity.

We construct the following sequence $\{f_k:\Lambda^k\to\RR\}_{k=1}^\infty$ of functions.
The function $f_0$ is taken to be $f$ on $\Lambda^0=\Lambda$. We take $f_k$ to be $$f_k(x,y)=\min(f_{k-1}(x,y), n_k(i_kx+j_ky)-\delta) \text{\ on\ } \Lambda^k = \Lambda^{k-1}\cap \{n_k(i_kx+j_ky)-\delta\geq 0\}.$$ Because of the choice of $n_k$ we know that $f_k$ and $f_{k-1}$ are equal outside of a small neighborhood of $O$. The number $n_k$ represents the quasi-degree of $f_n, n>k$ on the side dual to the vector $n_k(i_k,j_k)$. By Lemma~\ref{lemma_makingnice} for large $k$ all this $n_k$ can be chosen to be $1$. Therefore from the construction it is clear that $f_n$ is nice on $\Lambda^n$ for some $n$ big enough.
\end{proof}

\begin{proposition}
\label{prop_nice} Let $\Delta$ be a $\QQ$-polygon. Consider a sequence of operators $G_{{\bf q}_1},G_{{\bf q}_2},\dots, G_{{\bf q}_m}$ where ${\bf q}_1,{\bf q}_2,\dots, {\bf q}_m$ are  (not necessary distinct) points in $\Delta^\circ$. We will use the following notation 
\begin{equation}
\label{eq_compositionG}
G=G_{{\bf q}_m}G_{{\bf q}_{m-1}}\dots G_{{\bf q}_1}.
\end{equation}
 Then, for each $\e>0$ small enough there exists a unimodular $\QQ$-polygon $\Delta'\subset\Delta$ such that 
\begin{itemize}
\item $G 0_{\Delta'}$ is nice (Definition~\ref{def_nice}) on $\Delta'$,
\item $0\leq G 0_{\Delta}-G 0_{\Delta'}<\e$ on $\Delta'$. 
\item $G 0_{\Delta}\leq \e$ on $\Delta\setminus\Delta'$.
\end{itemize}
\end{proposition}

\begin{proof}
Consider $f=G0_\Delta$. Using Lemma~\ref{lemma_nicest}, we make necessary blow-ups at each corner of $\Delta$, constructing in this way a $\QQ$-polygon $\Delta'\subset \Delta$ and a nice function $\tilde f$ on $\Delta'$. By construction, $f=\tilde f$ near $P$. Therefore $G0_{\Delta'}\leq \tilde f$ and hence $G0_{\Delta'}$ is nice on $\Delta'$. Clearly $ G 0_{\Delta}\geq  G 0_{\Delta'}$ and we might do blow-ups in so small neighborhood of the corners of $\Delta$ such that $ G 0_{\Delta}<\e$ on $\partial\Delta'$ which implies the third assessment. The second assessment follows from Lemma~\ref{lem_distance}, because $\rho(0_{\Delta}, 0_{\Delta'})$ is arbitrary small for small $\e$, if these function are written in the canonical form. 
\end{proof}

\section{Coarse smooth approximation of the dynamic $G_P$}
\label{sec_coarse}

Let $\Delta$ be a $\QQ$-polygon and $g$ be a nice (Definition~\ref{def_nice}) $\Delta$-tropical series, such that $C(g)$ is a smooth tropical curve. Let $\q_i\in \Delta^\circ, i=1,\dots m$, and $f=G_{{\bf q}_m}G_{{\bf q}_{m-1}}\dots G_{{\bf q}_1}g$. Since each $G_{{\bf q}_k}$ is the application of $\Add_{i_k,j_k}^{e_{k}}$ for some $e_k>0$, we can write 

\begin{equation}
\label{eq_Gp}
G_{{\bf q}_m}G_{{\bf q}_{m-1}}\dots G_{{\bf q}_1}g = \Add_{i_mj_m}^{e_{m}}\Add_{i_{m-1}j_{m-1}}^{e_{m-1}}\dots \Add_{i_1,j_1}^{e_{1}}g.
\end{equation}

Suppose that the quasi-degrees $m_f,m_g$ coincinde (in particular, $f$ is also nice  on $\Delta$). For constants $M,h>0$ we
replace in  \eqref{eq_Gp} $$G_{{\bf q}_k}=\Add_{i_kj_k}^{e_k} \text{\ by\ } G_{{\bf q}_k}^\circ:=\Add_{i_kj_k}^{e_k-Mh} \text{\ for\ } k=1,\dots, m.$$ 

\begin{proposition}
\label{prop_smooth}
Denote $f_0=g, f_{k+1}=G_{{\bf q}_k}^\circ(f_k)$. Then, for each $\e>0$, there exists a constant $M$ such that for any $h>0$ small enough
\begin{itemize} 
\item all the tropical curves defined by $f_k,k=1,\dots, m$ are smooth or nodal (Definition~\ref{def_nodal}) on $\Delta$ as well as each tropical curve in the family during the application of $G_{{\bf q}_k}^\circ$ to $f_k$ (Remark~\ref{rem_smooth});
\item the tropical curve defined by $f_m$ is $\e$-close to the tropical curve defined by $G_{{\bf q}_m}G_{{\bf q}_{m-1}}\dots G_{{\bf q}_1}g$.
\end{itemize}
\end{proposition}

\begin{proof}
\label{proof_propsmooth}
Since $m_f=m_g$, we do not apply operators $G_{\q_i}$ in the regions adjacent to the boundary of $\Delta$. The only two possibilities \m{classification of cases} how the tropical curve can become non-smooth during our procedure in \eqref{eq_Gp} is appearance of a non-smooth vertex inside $\Delta^\circ$ and appearance of an edge with weight bigger than one inside $\Delta^\circ$ or at the corners of $\Delta$. 

To satisfy $\e$-closeness, it is enough that $mMh<\e$. It follows from Lemma~\ref{lemma_smooth} that a non-smooth vertex or an edge with weight bigger than one in $\Delta^\circ$ can appear only by contracting a face.  We can decrease the constants $e_i$ in \eqref{eq_Gp} by any small positive numbers, such that no $G_{{\bf q}_k}^\circ$ contracts a face, this eliminates a part of \m{nodedolali Gp for that no singular curves appear} the problems with smoothness inside $\Delta^\circ$. To be sure that this decreasing did not change the incidence between faces and points $\q_i$ in the process it is enough to choose $M$ such that $mMh$ (the total change of function) would be less than the minimal non-zero distance between one of the points $q_1,\dots, q_m$ and the tropical curves $G_{{\bf q}_i}\dots G_{{\bf q}_1}g, i=1,\dots,m$. 
Finally, $f_m$ is nice on $\Delta$ and, by Lemma~\ref{lemma_notnice}, the tropical curve $C(f_m)$ has no edges of weight bigger than one at the corners of $\Delta$.
\end{proof}

\section{Tropical symplectic area}
\label{sec_area}
One may ask what are intrinsic properties of $f_{\Omega,P}$. We will prove that the curve $C(f_{\Omega,P})$ solves a sort of Steiner problem, see Corollary~\ref{cor_sym}.
\begin{definition}[See \cite{tony}]\label{def_SArea}
The {\it tropical symplectic area} of an interval $L\subset\RR^2$ with a rational slope is $Area(L)=||L||\cdot ||v||,$ where $||-||$ denotes a Euclidean length and  $v$ is a primitive integer vector parallel to $l$. If $C$ is an $\Omega$-tropical curve, then its tropical symplectic area is the weighted sum of areas for its edges $e$, i.e. $$Area(C)=\sum_{e}Area(e)\cdot m_e,$$
where $m_e$ is the weight of the edge $e$ (Definition~\ref{def_smoothvertex}). This area may be infinite as well, if $C$ contains infinite number of edges and the series diverges or $C$ has edges of infinite length.
\end{definition}

The motivation for this definition is as follows. Recall that an amoeba of an algebraic curve $S$ in the algebraic torus
$(\mathbb{C}^*)^2$ is an image of $S$ in $\mathbb{R}^2$ under the
logarithm map $\log_t(z_1,z_2)=(\log_t |z_1|,\log_t
|z_2|)$. Consider a family $\{S_t\}$ of algebraic curves  in
$(\mathbb{C}^*)^2$ for $t>0$. We say that $\{S_t\}$
tropicalizes to  the tropical curve $C$ if the family $\log_t S_t\subset\mathbb{R}^2$ converges to $C$ when $t\to\infty$. It could seem that the tropicalization of $\{S_t\}$ is defined
only as a set. In fact, the multiplicities  for the edges of $C$ can
be also canonically restored from the family $S_t$. 

Consider the following symplectic form on $(\mathbb{C}^*)^2$:
 $$
\omega= -id\log(z_1)\wedge d\log(\bar z_1)-id\log(z_2)\wedge d\log(\bar z_2).$$
\begin{proposition}\label{prop_LimforSArea}
Let $C$ be the tropicalization for $\{S_t\}$ and $B$ be a convex bounded open subset of $\mathbb{R}^2$. Then 
$$\lim_{t\to\infty}\int\limits_{\log_t^{-1}(B)\cap S_t} \frac{1}{\log t}\omega = 4\pi^2 Area(C\cap B).$$ \end{proposition}
 
This justifies the name `` tropical symplectic area'': it is the main part in the asymptotic for symplectic areas.
 
 \begin{proof}
For a large $t$, $\log_t(S_t)$ is in a small neigborhood of the tropical curve $C$. Moreover, $S_t$ itself will be close to a certain lift of $C$ to the torus $(\mathbb{C}^*)^2.$ It is performed by lifting each edge with a slope $(p,q)$ to a piece of holomorphic cylinder $\{(z^p,z^q)|z\in\mathbb{C}\}$ translated by the action of the torus. This lift is called a complex tropical curve (see \cite{mikh} for the details). 

Therefore, we can compute the area of $S_t$ near the limit by looking at the areas of the cylinders. There also can be minor corrections coming from the vertices of $C$ but the corrections are small with respect to $\log t$ and so do not appear in the final statement.

To complete the proof we need to compute the contribution from each edge in $C\cap B$. It is clear that for each segment in $C\cap B$ the area of its lift is proportional to the length of the segment. So if we show that the area of the lift for the interval going from the origin to the integer vector $(p,q)$ is equal to $4\pi^2(p^2+q^2)\log t$ then we will be done. This computation is given by application of the following lemma for both parts of $\omega$.
 \end{proof}
 
 \begin{lemma}
 Let $v=(p,q)$ be a primitive integer vector. Let $C_t^{pq}$ be a lift of an interval $[0,v]$ to the torus $(\mathbb{C}^*)^2$ under $\log_t,$ i.e. $C_t^{pq}=\{(z^p,z^q)| 1\leq|z|\leq t\}.$ Then 
 $$\int_{C_t^{pq}} d\log(z_1)\wedge d\log(\bar z_1)=-4i\pi^2p^2\log t.$$
 \end{lemma}
 
 \begin{proof} Let $z_1$ be $r\exp(i\phi),$ where $r>0$ and $\phi\in[0,2\pi].$  Then $$d\log z_1=d\log\, r+i\phi d\phi\text{\ and\ }$$ $$d\log z_1\wedge d\log \bar{ z_1}= -id\log r\wedge d\phi^2$$ Then the left hand side of the equality we are proving is equal to 
 $$-i\int_1^t\int_0^{2\pi}d\log r\wedge d\phi^2=-4i\pi^2p^2\log t.$$\end{proof} 

\begin{remark}
The specific choice for $\omega$ is not crucial while it is invariant under the action of $(C^*)^2$. Indeed, if $\omega'$ is an arbitrary 2-form then its restriction to any holomorphic curve will not have contributions from pure holomorphic and anti-holomorphic parts of $\omega'.$ So we can think that $\omega'$ is a $(1,1)$-form. There is a two dimensional family of torus-invariant $(1,1)$-forms. Different choices for $\omega$ from this family correspond to coordinate dilatations on the level of tropical curves. 
\end{remark}

Proposition \ref{prop_LimforSArea} suggests us that symplectic area for tropical curves should be deformation invariant. Indeed, this should follow from the fact that the $2$-form $\omega$ is closed. 
And indeed, we can prove the deformation invariance directly.

 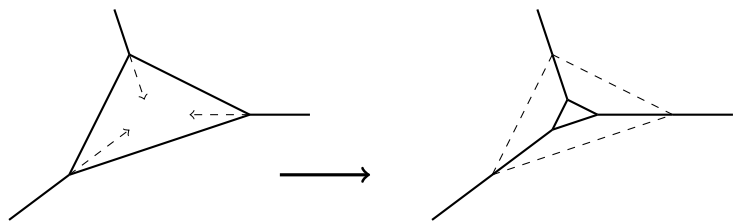
\begin{figure}[h]
    \centering

\begin{tikzpicture}[scale=0.4]

\draw[thick](0,0)--++(-2,-1.5);
\draw[thick](0,0)--++(2,4);
\draw[thick](0,0)--++(6,2);
\draw[thick](6,2)--++(-4,2);
\draw[thick](6,2)--++(2,0);
\draw[thick](2,4)--++(-0.5,1.5);
\draw[->][dashed](0,0)--++(2,1.5);
\draw[->][dashed](6,2)--++(-2,0);
\draw[->][dashed](2,4)--++(0.5,-1.5);

\draw[->][very thick](7,0)--(10,0);

\begin{scope}[xshift=400]
\draw[thick](-2,-1.5)--++(4,3);
\draw[thick](2,1.5)--++(0.5,1);
\draw[thick](2,1.5)--++(1.5,0.5);
\draw[thick](3.5,2)--++(-1,0.5);
\draw[thick](6,2)--++(2,0);
\draw[thick](6,2)--++(-2.5,0);
\draw[thick](2,4)--++(0.5,-1.5);
\draw[thick](2,4)--++(-0.5,1.5);
\draw[dashed](0,0)--++(2,4);
\draw[dashed](0,0)--++(6,2);
\draw[dashed](6,2)--++(-4,2);
\end{scope}

\end{tikzpicture}

    \caption{In the picture we shrink a triangular cycle. Any deformation of a tropical curve can be decomposed into such operations or their inversions. } 
    \label{fig_LocalDeform}
\end{figure}

\begin{lemma}
\label{lemma_deform}
Consider a continuous  family $C_s$ tropical curves such that for a compact set $B\subset\RR^2$ we have that $C_s\setminus B$ does not depend on $s$. Then $Area(C_s)$ is constant.
\end{lemma}

\begin{proof}

Any deformation $C_s$ locally can be decomposed into the elementary ones. Near each vertex of $C$, an elementary deformation is a process of moving and shortening two edges while growing the one in the opposite direction (see Figure \ref{fig_LocalDeform}).
  
Globally this corresponds to enlarging a coefficient for a tropical polynomial. For example on Figure \ref{fig_ShrinkPhi} we change the coefficient for the central region.

Up to a scaling, an elementary deformation simply replaces the union of segments $[0,v_1]$ an $[0, v_2]$ 
by a single segment $[0,v_1+v_2]$. Here $v_1$ and $v_2$ are the primitive (or appropriate multiples of primitive) vectors for the edges we are moving. 
Denote by $w_i$ the projection of $v_1+v_2$ on the line spanned by $v_i$ (see Figure~\ref{ChangeOfArea}). Then after the deformation the two edges together loose $$|v_1||w_1|+|v_2||w_2|=|v_1|(v_1+v_2)\cdot {{v_1}\over {|v_1|}}+|v_2|(v_1+v_2)\cdot {{v_2}\over {|v_2|}}=|v_1+v_2|^2$$ of their tropical symplectic area. On the other hand, the growing edge contributes exactly $|v_1+v_2|^2$ to the symplectic area of the deformed curve.
\end{proof}

 \begin{figure}[h]
    \centering

\begin{tikzpicture}[scale=0.4]

\draw[thick](0,0)--++(2,4);
\draw[thick](0,0)--++(6,2);
\draw[thick](0,0)--++(8,6);
\draw(8,6)node[right]{$v_1+v_2$};
\draw(2,4)node[left]{$v_2$};
\draw(6.5,1.5)node{$v_1$};
\draw(0,0)node[left]{$0$};
\draw[dashed](6,2)--++(6,2);
\draw[dashed](2,4)--++(3,6);
\draw[dashed](8,6)--++(-4,2);
\draw[dashed](8,6)--++(1,-3);
\draw[dashed](8,6)--++(-2,-4);
\draw[dashed](8,6)--++(-6,-2);
\draw(9.5,2.5)node{$w_1$};
\draw(4,8)node[left]{$w_2$};
\end{tikzpicture}

    \caption{Computing contributions for symplectic area.} 
    \label{ChangeOfArea}
\end{figure}
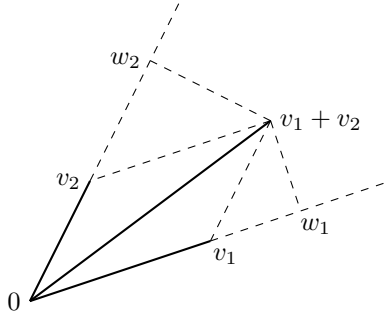

Let us get back to our specific case. Let $\Delta$ be a compact $\QQ$-polygon and $f$ be a $\Delta$-tropical polynomial with quasidegree $m_f$ (Definition~\ref{def_qdegree}). Then we can deform $C(f)$ to the union of all edges $e$ of the polygon taken with the multiplicities $m_f(e).$ This observation together with the deformation-invariance (Lemma~\ref{lemma_deform}) proves the following lemma. 

\begin{lemma}\label{lemma_quasidegdefinesarea}
Under the above assumptions $Area(C(f))=\sum\limits_{e\in S(\Delta)} m_f(e) Area(e).$
\end{lemma}

\begin{corollary}[cf. Theorem 3 in \cite{announce}]
\label{cor_sym}
If $\Delta$ is a compact $\QQ$-polygon and $P\subset \Delta^\circ$ is a finite collection of points, then the tropical curve $C(f_{\Delta,P})$ has the minimal tropical symplectic area among all $\Delta$-tropical curves passing through the configuration of points $P$.
\end{corollary}
Indeed, the tropical symplectic area is determined by the quasidegree, and $f_{\Omega,P}$ has the minimal on each side of $\Delta$ degree among the $\Delta$-tropical series non-smooth at $P$.

We should mention that the tropical symplectic area of a tropical curve already appeared in physics under the name of {\it mass} of a web \cite{kol1998bps}, where a {\it web} is a direct analog of a tropical curve. Only for curiosity we present a part of the dictionary between tropical objects and the field theory.
We quote \cite{strings}: ``On the other hand, we already know
that shrinking an internal face of a string web corresponds to a bosonic zero mode ... so the mass of
the web is independent of this deformation.'' -- this reminds us the operator $G_\p$, shrinking a face. This mass is also presented as a trace of some operator (BPS-formula, {\it ibidem}), and the area of a face in $\Delta\setminus C(f)$ is interpreted there as the ``tension of a monopolic string''.

\section{Summary}
\label{sec_summary}
For easy reference we formulate here a theorem, which summarizes most things about coarsening that we need in \cite{us}.

\begin{theorem}
\label{lemma_super}
Choose $\e>0$. For a given $\QQ$-polygon $\Delta$ and a finite set $P\subset \Delta^\circ$ there exist a $\QQ$-polygon $\Delta'\subset\Delta$ and a $\Delta'$-tropical polynomial $g$ such that 
\begin{enumerate}[a)]
\item $g|_{\Delta'}<\e$, the curve $C(g)$ is smooth, and $G_P g=G_P 0_{\Delta'}$,
\item $G_P g$ is $\e$-close to $G_P 0_{\Delta}$. 

Using Proposition~\ref{prop_gpsconvergence} let us write $G=G_{{\bf q}_m}G_{{\bf q}_{m-1}}\dots G_{{\bf q}_1}g$ such that $G$ is $\e$-close to $G_P g$ and their quasi-degrees (Definition~\ref{def_qdegree}) coincide. 
\item Then, during the calculation of $G$ we never apply a wave operator for a face which has a common side with $\partial\Delta'$.

Note that in the product $G_{{\bf q}_m}G_{{\bf q}_{m-1}}\dots G_{{\bf q}_1}g$ each $G_{{\bf q}_k}$ is the application of $\Add_{i_k,j_k}^{e_{k}}$ for some $e_k>0$, i.e. we increase the coefficient in the monomial $i_kx+j_ky$ by $e_k$. So we have 

\begin{equation}
\laber{eq_Gp}
G_{{\bf q}_m}G_{{\bf q}_{m-1}}\dots G_{{\bf q}_1}g = \Add_{i_mj_m}^{e_{m}}\Add_{i_{m-1}j_{m-1}}^{e_{m-1}}\dots \Add_{i_1j_1}^{e_{1}}g.
\end{equation}
For a constant $M$ we
replace in  \eqref{eq_Gp} $$G_{{\bf q}_k}=\Add_{i_kj_k}^{e_k} \text{\ by\ } G_{{\bf q}_k}^\circ:=\Add_{i_kj_k}^{e_k-M\h} \text{\ for\ } k=1,\dots, m.$$ Denote $f_0=g, f_{k+1}=\Add_{i_kj_k}^{e_k-M \h}f_k =G_{{\bf q}_k}^\circ(f_k)$. 
\item Then there exists a constant $M$ such that for any $\h>0$ small enough all the tropical curves defined by $f_k,k=1,\dots, m$ are smooth or nodal (Definition~\ref{def_nodal}) on $\Delta$ as well as each tropical curve in the family during the application of $G_{{\bf q}_k}^\circ$ to $f_k$ (Remark~\ref{rem_smooth}); and
\item the tropical curve defined by $f_m$ is $\e$-close to the tropical curve defined by $G_{{\bf q}_m}G_{{\bf q}_{m-1}}\dots G_{{\bf q}_1}g$.
\end{enumerate}
\end{theorem}

Theorem~\ref{th_verge} gives a), b). Then, c) follows from the fact that the quasi-degrees of  $G, G_P g$ coincide. The content of Proposition~\ref{prop_smooth} is d),e).

\bibliography{../../sand/sandbib}
\bibliographystyle{abbrv}

Nikita Kalinin, \email{nikaanspb\{at\}gmail.com}

Mikhail Shkolnikov, \email{mikhail.shkolnikov\{at\}gmail.com}

\end{document}